\documentclass[11pt]{amsart}
\usepackage[margin=1in]{geometry}
\usepackage{amsmath,amssymb,amsthm,booktabs,array}
\usepackage{mathtools}
\usepackage{tikz}
\usepackage[hypertexnames=false]{hyperref}
\allowdisplaybreaks
\numberwithin{equation}{section}

\newtheorem{theorem}{Theorem}[section]
\newtheorem{proposition}[theorem]{Proposition}
\newtheorem{lemma}[theorem]{Lemma}
\newtheorem{corollary}[theorem]{Corollary}
\newtheorem{hypothesis}[theorem]{Hypothesis}

\newtheorem*{conjecture*}{Conjecture}
\theoremstyle{remark}
\newtheorem{remark}[theorem]{Remark}

\newcommand{\Z}{\mathbb{Z}}
\newcommand{\R}{\mathbb{R}}
\newcommand{\N}{\mathbb{N}}

\newcommand{\eps}{\varepsilon}
\newcommand{\Kr}[2]{\genfrac{(}{)}{}{}{#1}{#2}}
\newcommand{\ind}{\mathbf 1}

\title[The large sieve for square moduli under Hooley's hypothesis $R^*$]{An improvement of the large sieve for square moduli under Hooley's hypothesis $R^{*}$ for short Sali\'e sums}
\author{Stephan Baier}
\address{Stephan Baier, Ramakrishna Mission Vivekananda Educational and Research Institute, Department of Mathematics, G.~T.\ Road, PO Belur Math, Howrah, West Bengal 711202, India}
\email{stephanbaier2017@gmail.com}
\subjclass[2020]{Primary 11N35; Secondary 11L05, 11L07, 11L40}
\keywords{Large sieve with square moduli, quadratic Gauss sums, Sali\'e sums, Kloosterman sums, Hooley's hypothesis $R^*$}
\date{}

\begin{document}
\begin{abstract}
Let
\[
S(Q,M,N,(a_n))=\sum_{q\le Q}\ \sum_{\substack{a\bmod q^2\\(a,q)=1}}\Big|\sum_{M<n\le M+N}a_n\,e\Big(\frac{an}{q^2}\Big)\Big|^2
\]
be Zhao's large sieve sum with square moduli. Zhao conjectured that $S\ll(QN)^\eps(Q^3+N)\sum|a_n|^2$. The bound of Baier and Zhao (2008) exceeds this conjecture by a factor which reaches $Q^{1/2}$ at the critical point $N=Q^3$, and there this exponent $\tfrac12$ has not been lowered since. We prove that, under Hooley's Hypothesis $R^*$ for short Sali\'e sums,
\[
S(Q,M,N,(a_n))\ \ll\ (QN)^\eps\big(Q^3+N+Q^{4/3}N^{2/3}\big)\sum|a_n|^2 .
\]
This replaces the excess over Zhao's conjecture by its $\tfrac23$-th power throughout the range $Q^{5/2}<N<Q^4$; at $N=Q^3$ it gives $S\ll Q^{1/2-1/6+\eps}N\sum|a_n|^2$. The proof is short. Following Zhao's original treatment, we approximate the point under consideration by fractions $b/r$ with denominators up to $Q^{2/3}N^{1/3}$, which for $N<Q^4$ is beyond the customary $\sqrt N$. If the approximation is poor, an elementary count suffices. If it is good, the analytic part of the problem disappears: after Poisson summation in both variables the count is expressed through quadratic Gauss sums, and their exact evaluation at every modulus turns the sum over the dual variable into short Kloosterman and Sali\'e sums. Hypothesis $R^*$ then yields the square root of the elementary count. At odd moduli the hypothesis is needed only in the form of Conjecture (C$^*$) of Friedlander and Iwaniec. The Gauss-sum evaluations, including even moduli and coefficients sharing a factor with the modulus, are proved in full. This version supersedes earlier versions of this preprint, which obtained the exponent $\tfrac12-\tfrac1{28}$ at $N=Q^3$ by a much longer route through sums of modular square roots. The paper was prepared in collaboration with Claude (Anthropic); Section \ref{sec:contributions} sets out what each of us contributed.
\end{abstract}
\maketitle

\tableofcontents

\section{Introduction}\label{sec:intro}

\subsection{The large sieve with square moduli}\label{sec:background}

Let $a_n$, $M<n\le M+N$, be complex numbers, and put $Z:=\sum_n|a_n|^2$. The classical large sieve inequality states that
\[
\sum_{q\le Q}\ \sum_{\substack{a=1\\(a,q)=1}}^{q}\ \bigg|\sum_{M<n\le M+N}a_n\,e\Big(\frac aq\,n\Big)\bigg|^2\ \ll\ (Q^2+N)Z ,
\]
which is optimal up to constants. Zhao \cite{Zhao2004} introduced an arithmetically different variant, in which the denominators are restricted to \emph{perfect squares}:
\[
S(Q,M,N,(a_n)):=\sum_{q\le Q}\ \sum_{\substack{a=1\\(a,q)=1}}^{q^2}\ \bigg|\sum_{M<n\le M+N}a_n\,e\Big(\frac a{q^2}\,n\Big)\bigg|^2 .
\]
The points $a/q^2$ with $q\le Q$ are about $Q^3$ in number, far fewer than the Farey fractions of denominator up to $Q^2$, and they are irregularly spaced; the classical large sieve, which is at bottom a statement about how many of the points can lie close to a given real number, therefore gives only a weak bound. Zhao conjectured
\begin{equation}\label{eq:zhaoconj}
S(Q,M,N,(a_n))\ \ll_\eps\ (QN)^\eps\big(Q^3+N\big)Z ,
\end{equation}
which matches a naive count of the points, and this remains open. The bound of Baier and Zhao \cite{BaierZhao2008} is
\begin{equation}\label{eq:BZ}
S(Q,M,N,(a_n))\ \ll_\eps\ (QN)^\eps\Big(Q^3+N+\min\big\{Q^2N^{1/2},\,Q^{1/2}N\big\}\Big)Z .
\end{equation}
It agrees with \eqref{eq:zhaoconj} for $N\le Q^{5/2}$ and for $N\ge Q^4$. In between it exceeds the conjecture by the factor
\begin{equation}\label{eq:BZexcess}
\frac{N}{Q^{5/2}}\quad(Q^{5/2}\le N\le Q^3),\qquad\qquad \Big(\frac{Q^4}N\Big)^{1/2}\quad(Q^3\le N\le Q^4),
\end{equation}
which is largest, equal to $Q^{1/2}$, at the \emph{critical point} $N=Q^3$. For higher power moduli $q^k$ better bounds are known in parts of the range: Baker, Munsch and Shparlinski \cite{BMS} used additive energy, and Baier and Lynch \cite{BaierLynch2024}, using results on Waring's problem, obtained the strongest bounds currently known in large ranges of the parameters for every $k\ge4$ (see also the references in \cite{Baier2026b}). For square moduli, \eqref{eq:BZ} has not been improved unconditionally at the critical point. There Zhao's conjecture predicts $S\ll Q^\eps NZ$, while \eqref{eq:BZ} gives
\[
S(Q,M,Q^3,(a_n))\ \ll_\eps\ Q^{1/2+\eps}NZ ,
\]
and no unconditional improvement of this exponent $\tfrac12$ has been found in the nearly two decades since \cite{BaierZhao2008}.

\subsection{The results}\label{sec:mainresults}

We assume Hooley's Hypothesis $R^*$ for short Sali\'e sums, in the form of Hypothesis \ref{hyp:R*} below: square-root cancellation, up to the customary losses, in
\[
\sum_{x_1<n\le x_2}\Kr nc\,e_c\big(a\bar n+bn\big)\qquad\big(0\le x_2-x_1\le c\big),
\]
where $\Kr nc$ is the Kronecker symbol and $\bar n$ the inverse of $n$ modulo $c$. Section \ref{sec:salie} explains how this hypothesis relates to the hypotheses of Hooley and to the conjectures of Friedlander and Iwaniec.

\begin{theorem}[Main theorem]\label{thm:main}
Assume Hypothesis \ref{hyp:R*}. Then for all $Q,N\ge1$, $M\in\R$ and all complex $a_n$,
\begin{equation}\label{eq:main}
S(Q,M,N,(a_n))\ \ll_\eps\ (QN)^\eps\Big(Q^3+N+Q^{4/3}N^{2/3}\Big)Z .
\end{equation}
\end{theorem}

The new term $Q^{4/3}N^{2/3}$ is at most $Q^3$ for $N\le Q^{5/2}$ and at most $N$ for $N\ge Q^4$, so \eqref{eq:main} is Zhao's conjecture there, as \eqref{eq:BZ} already is. In the range $Q^{5/2}<N<Q^4$ the excess over the conjecture is
\[
\Big(\frac{N}{Q^{5/2}}\Big)^{2/3}\quad(Q^{5/2}\le N\le Q^3),\qquad\qquad \Big(\frac{Q^4}N\Big)^{1/3}\quad(Q^3\le N\le Q^4),
\]
which is exactly the $\tfrac23$-th power of the excess \eqref{eq:BZexcess} in \eqref{eq:BZ}. At the critical point:

\begin{corollary}\label{cor:critical}
Assume Hypothesis \ref{hyp:R*}. Then
\[
S(Q,M,Q^3,(a_n))\ \ll_\eps\ Q^{1/2-1/6+\eps}\,NZ .
\]
\end{corollary}

As usual, the large sieve bound is reduced to a counting problem. For $N\ge1$ put $\Delta:=1/N$ and, for $\alpha\in\R$,
\begin{equation}\label{eq:Pdef}
P(\alpha):=\#\Big\{(q,a)\in\Z^2:\ 1\le q\le Q,\ (a,q)=1,\ \Big|\frac a{q^2}-\alpha\Big|\le\Delta\Big\},
\end{equation}
the number of fractions with square denominator within $\Delta$ of $\alpha$. By Lemma \ref{lem:LS}, $S$ is at most a constant times $NZ$ times the maximum of $P(\alpha)$ over $\alpha$, and Theorem \ref{thm:main} follows from:

\begin{theorem}\label{thm:P}
Assume Hypothesis \ref{hyp:R*}. Then for all $Q,N\ge1$, uniformly in $\alpha\in\R$,
\begin{equation}\label{eq:Pbound}
P(\alpha)\ \ll_\eps\ (QN)^\eps\Big(1+\frac{Q^3}N+\frac{Q^{4/3}}{N^{1/3}}\Big).
\end{equation}
In particular $P(\alpha)\ll Q^{1/3+\eps}$ when $N=Q^3$.
\end{theorem}

The bound $P(\alpha)\ll Q^{1/2+\eps}$ at $N=Q^3$ is what an elementary argument gives, and it is the source of the exponent $\tfrac12$ in \eqref{eq:BZ}.

\subsection{The idea of the proof}\label{sec:idea}

\emph{The setting.} Let $\alpha\in\R$. By Dirichlet's theorem, with a parameter $\tau\ge1$ to be chosen, there is a fraction $b/r$ with
\[
\alpha=\frac br+z,\qquad (b,r)=1,\qquad 1\le r\le\tau,\qquad |z|\le\frac1{r\tau}.
\]
A pair $(q,a)$ counted by $P(\alpha)$ gives the integer $m:=ar-bq^2$, which satisfies
\[
m\equiv-bq^2\pmod r,\qquad |m-q^2rz|\le q^2r\Delta .
\]
So $m$ lies in a window of width $2q^2r\Delta$ centred at $q^2rz$; as $q$ varies, the window moves over a range of size up to $Q^2r|z|$.

\emph{The elementary count.} For each $m$, the congruence confines $q$ to $\rho_r(-\bar bm)$ residue classes modulo $r$, where $\rho_r(t)$ is the number of square roots of $t$ modulo $r$, and the window condition confines $q$ to an interval. Counting at most the expected number plus one element in each class gives, \emph{for every $r$ and every $z$} (Lemma \ref{lem:E}),
\begin{equation}\label{eq:ideaE}
P(\alpha)\ll\Big(1+Q^2r\big(|z|+\Delta\big)+\frac{Q^3}N\Big)(QN)^\eps .
\end{equation}
The term $Q^2r(|z|+\Delta)$ is essentially the number of admissible values of $m$, and its two parts behave quite differently. The part $Q^2r|z|$ comes from the movement of the window and is controlled by Dirichlet's theorem: $r|z|\le1/\tau$ for every $r$. The part $Q^2r\Delta$ comes from the width of the window. It does not depend on $z$ at all: it is as large for $z=0$ as for $|z|=\Delta$, and it grows with $r$.

\emph{The classical choice $\tau=\sqrt N$.} Since $r|z|\le1/\tau$ and $r\le\tau$, \eqref{eq:ideaE} gives $P(\alpha)\ll(1+Q^2/\tau+Q^2\tau/N+Q^3/N)(QN)^\eps$, and the best choice is $\tau=\sqrt N$. Then both parts are $Q^2/\sqrt N$: the first for a poor approximation with $r|z|\asymp1/\sqrt N$, the second for a good approximation, even with $z=0$, with $r\asymp\sqrt N$. At $N=Q^3$ this is $Q^{1/2}$, so the exponent $\frac12$ in \eqref{eq:BZ} at the critical point is exactly what the elementary count gives with this choice. In the earlier treatments \cite{Baier2006,BaierZhao2008}, and in the earlier versions of this paper, $\tau=\sqrt N$ was kept, and an improvement was sought where $|z|$ is large. That is where the Fourier-analytic treatment is hardest: the window for $m$ moves with $q$, and an exponential integral with a quadratic phase oscillates.

\emph{The new choice: let the denominators grow.} We take $\tau$ larger than $\sqrt N$. This changes where the difficulty lies.
\begin{itemize}
\item \emph{For poor approximations, large denominators do no harm.} If $|z|>\Delta$, then $r|z|\le1/\tau$ for every $r\le\tau$, and \eqref{eq:ideaE} gives $P(\alpha)\ll(1+Q^2/\tau+Q^3/N)(QN)^\eps$, which only improves as $\tau$ grows.
\item \emph{The difficulty is confined to good approximations with large denominators.} If $|z|\le\Delta$, then \eqref{eq:ideaE} gives $Q^2r\Delta$, which is too large once $r$ is large, even for $z=0$.
\item \emph{In exactly this case the analytic term does not oscillate.} In Zhao's treatment one counts the $a$ for each $q$ with a smooth weight of length $H\asymp Q^2\Delta$, centred at $\alpha q^2$. Poisson summation in $a$ produces frequencies $|h|\ll Q^\eps/H$ and the factor $e(hzq^2)$, whose phase satisfies $|hzq^2|\ll Q^\eps|z|/\Delta$. So the quadratic phase oscillates only when $|z|$ is large compared with $\Delta$, and that is the case handled by the elementary count (Remark \ref{rem:zhao}, which is heuristic and not used in the proofs). In our normalization $z$ disappears altogether: since $|z|\le\Delta$, we have $|m|\le2q^2r\Delta$ for every counted pair, so $m$ lies in a \emph{fixed} window around $0$, whatever $q$ is. What remains is a purely arithmetic problem.
\end{itemize}
We majorize the count by a smooth sum over all $q$ in a dyadic range and all $m$ in the window with $m\equiv-bq^2\ (r)$, and apply Poisson summation in $q$ and in $m$, both modulo $r$. This produces the quadratic Gauss sums
\[
G(-hb,l,r)=\sum_{k\bmod r}e_r\big({-}hbk^2+lk\big),
\]
where $l$ and $h$ are the variables dual to $q$ and $m$. The Gauss sums are then evaluated exactly, at every modulus (Section \ref{sec:gauss}). For fixed $l$, the sum over $h$ becomes a short sum of Kloosterman or Sali\'e type,
\[
\sum_{h'}\widehat V\Big(\frac{h'}T\Big)\Kr{h'}{s}\psi(h')\,e_{s}\big(A\overline{h'}\big),\qquad s=\frac r{(h,r)},\quad T=\frac s W ,
\]
where $W\asymp Q^2r/N$ is the length of the window for $m$ and $\psi$ is one of three fixed characters modulo $8$. Hypothesis $R^*$ gives square-root cancellation in it, and the outcome is
\[
P(\alpha)\ll\Big(1+\frac{Q^3}N+\Big(\frac{Q^2r}N\Big)^{1/2}\Big)(QN)^\eps
\]
(Proposition \ref{prop:F}): the \emph{square root} of the elementary count $1+Q^2r\Delta$ in this case. For good approximations with $Q^2r\Delta\le Q^{4/3}N^{-1/3}$, that is $r\le Q^{-2/3}N^{2/3}$, the elementary count \eqref{eq:ideaE} already suffices; Hypothesis $R^*$ is needed only for good approximations with larger denominators.

\emph{The balance.} A poor approximation now contributes $Q^2/\tau$, which decreases with $\tau$, and a good one at most $(Q^2\tau/N)^{1/2}$, which increases with $\tau$. They are equal for $\tau=Q^{2/3}N^{1/3}$, where both are $Q^{4/3}N^{-1/3}$, and Theorem \ref{thm:P} follows. The comparison with the classical choice is as follows:
\begin{center}
\renewcommand{\arraystretch}{1.3}
\begin{tabular}{lcc}
\toprule
 & poor: $|z|>\Delta$ & good: $|z|\le\Delta$\\
\midrule
$\tau=\sqrt N$; Lemma \ref{lem:E} in both cases & $Q^2/\sqrt N$ & $Q^2/\sqrt N$\\
$\tau=Q^{2/3}N^{1/3}$; Lemma \ref{lem:E}, Proposition \ref{prop:F} & $Q^2/\tau$ & $(Q^2\tau/N)^{1/2}$\\
\bottomrule
\end{tabular}
\end{center}
In short, the new choice of $\tau$ leaves the regime in which the analytic term oscillates to the elementary count, and leaves to Fourier analysis only the regime in which it does not, where the problem is arithmetic. At $N=Q^3$ the first row is $Q^{1/2}$ and the second $Q^{1/3}$; the new value $\tau=Q^{5/3}$ lies beyond the value $\tau=Q^{3/2}$ customary in this problem since \cite{Baier2006,BaierZhao2008}, and this is essential. Figure \ref{fig:balance} shows the three bounds as functions of $\tau$.

\begin{figure}[ht]
\centering
\begin{tikzpicture}[x=9cm,y=7cm,>=stealth]
\draw[->] (1.45,0) -- (2.06,0) node[right] {$t$};
\draw[->] (1.45,0) -- (1.45,0.66) node[above] {exponent};
\foreach \t/\l in {1.5/\tfrac32,1.6667/\tfrac53,2/2} \draw (\t,0.01) -- (\t,-0.01) node[below] {$\l$};
\foreach \y/\l in {0.3333/\tfrac13,0.5/\tfrac12} \draw (1.455,\y) -- (1.445,\y) node[left] {$\l$};
\draw[thick] (1.45,0.55) -- (2,0) node[pos=0.93,above right] {$2-t$};
\draw[thick,dashed] (1.45,0.45) -- (1.66,0.66) node[right] {$t-1$};
\draw[very thick] (1.45,0.225) -- (2,0.5) node[right] {$\frac{t-1}2$};
\draw[dotted] (1.5,0) -- (1.5,0.5);
\draw[dotted] (1.6667,0) -- (1.6667,0.3333);
\fill (1.5,0.5) circle (0.6pt);
\fill (1.6667,0.3333) circle (0.6pt);
\end{tikzpicture}
\caption{The bounds for $P(\alpha)$ at $N=Q^3$, as powers of $Q$, when $\tau=Q^t$: the elementary count for a poor approximation ($2-t$) and for a good one ($t-1$), and Proposition \ref{prop:F} for a good approximation ($\frac{t-1}2$). With the elementary count alone the best choice is $t=\frac32$, giving $Q^{1/2}$; with Proposition \ref{prop:F} it is $t=\frac53$, giving $Q^{1/3}$.}\label{fig:balance}
\end{figure}

\emph{The same balance in terms of continued fractions.} A poor approximation at one denominator is a good approximation at a larger one. If $b/r$ is the last convergent of $\alpha$ with $r\le Q^{3/2}$ and $|z|=Q^{-3/2-\gamma}r^{-1}$ with $\gamma\ge0$, the next convergent has denominator $r_+\asymp Q^{3/2+\gamma}$ and approximates $\alpha$ within $\Delta$. The elementary count at $b/r$ gives $Q^{1/2-\gamma}$, while at the next convergent it gives $Q^{1/2+\gamma}$ and Proposition \ref{prop:F} gives $Q^{1/4+\gamma/2}$. Taking the better of the two descriptions, the worst case is $\gamma=\tfrac16$, with $Q^{1/3}$. With $\tau=\sqrt N$ only the first description is available, and $\gamma=0$ gives $Q^{1/2}$. Section \ref{sec:remarks} discusses what limits the exponent $\tfrac13$.

\subsection{Why modular square roots appeared, and why they are not needed here}\label{sec:history-msr}

The problem was originally linked to modular square roots for the following reason. If, in the situation above, one sums over $q$ first, then the congruence $bq^2\equiv-m\ (r)$ says that $q$ runs over the square roots of $-\bar bm$ modulo $r$. Poisson summation in $q$ alone then turns the count into bilinear sums
\[
\sum_l\sum_m\alpha_l\beta_m\sum_{\substack{k\bmod r\\k^2\equiv jm\ (r)}}e_r(lk)\,e\big(lf(m)\big),\qquad j\equiv-\bar b\pmod r ,
\]
with a mildly oscillating phase $f$, in which the arithmetic is carried by the modular square roots of the integers $jm$ in a short interval. In the author's work \cite{Baier2026b,Baier2025} these bilinear sums are estimated through \emph{additive energies} of such sets of modular square roots, a natural object studied by Kerr, Shkredov, Shparlinski and Zaharescu \cite{KSSZ}, and \cite[Corollary 21]{Baier2026b} obtains $S\ll Q^{1/2-1/135+\eps}NZ$ at $N=Q^3$ from two hypotheses on such energies. This energy route is of independent interest, and the energy bounds it calls for are a problem in their own right; we do not follow it here.

The preprint \cite{Baier2026a} and the earlier versions of the present paper kept the modular square roots but treated the bilinear sums through Sali\'e sums: a short sum of modular square roots is, after completion and evaluation of quadratic Gauss sums, a Sali\'e sum of the dual length. That route led to the exponent $\tfrac12-\tfrac1{28}$ at $N=Q^3$, but only after a long unconditional treatment of small moduli. The present version sums over $m$ first. The Gauss sums then appear at once, and modular square roots do not appear at all. Section \ref{sec:versions} describes what has changed.

\subsection{What is proved here, and what is quoted}\label{sec:selfcontained}

Apart from Hypothesis \ref{hyp:R*}, which is not proved, the paper is self-contained, with three classical inputs:
\begin{itemize}
\item Gauss's theorem on the sign of the quadratic Gauss sum (Proposition \ref{prop:gauss1});
\item Dirichlet's approximation theorem;
\item the Poisson summation formula, for smooth compactly supported functions and for the Fej\'er kernel, together with the standard facts about the Fourier transforms involved (that of the Fej\'er kernel, and the rapid decay of those of smooth compactly supported functions).
\end{itemize}
Everything else is proved in full: the reduction of the large sieve to $P(\alpha)$ (Lemma \ref{lem:LS}), the elementary count (Lemma \ref{lem:E}), the evaluation of quadratic Gauss sums at every modulus, including powers of $2$ and coefficients sharing a factor with the modulus (Section \ref{sec:gauss}), the passage from Hypothesis \ref{hyp:R*} as stated to the characters that actually occur (Section \ref{sec:salieweights}), and the double Poisson summation (Proposition \ref{prop:F}). Neither Weil's bound nor Sali\'e's evaluation of complete Sali\'e sums is used in any proof; they appear only in the discussion of the hypothesis (Proposition \ref{prop:completeSalie}).

\subsection{Earlier versions of this paper}\label{sec:versions}

Earlier versions of this preprint (arXiv:2609.21195) proved, under the same hypothesis, $S\ll Q^{1/2-\eta+\eps}NZ$ at $N=Q^3$, first with $\eta=\tfrac1{134}$ and then with $\eta=\tfrac1{28}$. They worked with denominators $r\le\sqrt N=Q^{3/2}$. In that framework the analytic oscillation forces one to treat each $r$ by a separate method: the conditional argument covered only $r\ge Q^{1/2+\eps}$, and the size of the saving was decided by an unconditional treatment of the small moduli $r\le Q^{2/3}$, which occupied about half of the paper. The present argument does not need any of that, and those parts have been removed: the reduction to sums of modular square roots and the bilinear sum bound, the treatment of very small moduli by Cauchy--Schwarz over residue classes, and the second part of the paper, which sharpened the complete-sum input to the small-moduli argument of \cite[\S9]{Baier2026b} and recorded two corrections to that argument. The unconditional results of those parts are, to our knowledge, correct; they are not needed here, and they may appear separately. The section on Gauss sums is kept, essentially unchanged, because it is exactly what the present argument needs.

\subsection{Contributions, and the history of this work}\label{sec:contributions}

This paper was written by the author in collaboration with Claude (Anthropic), and it seems right to say plainly what each side contributed and how the work was done.

The mathematical programme is the author's: the problem, the decision to work under Hooley's Hypothesis $R^*$ for short Sali\'e sums rather than under hypotheses on additive energies, the insistence on every modulus rather than on prime and prime-square moduli, the choice of which lines of attack to pursue and which to abandon, the identification of the prior literature, and the judgement of what constitutes an acceptable proof. Claude carried out, under that direction, the detailed technical work: the evaluation of quadratic Gauss sums at every modulus, the double Poisson summation, the estimates, the parameter choices, the numerical checks, and the write-up.

The work went through several stages, and the last of them is the reason this version exists.

\emph{The starting point.} The preprint \cite{Baier2026a} proved the bound needed for the large sieve at prime and prime-square moduli $r$, under Hypothesis $R^*$, by Weyl differencing, and asked whether it holds for all $r\le Q^{3/2}$. The first version of this paper answered this, for every modulus, by Weyl differencing and a transference to complete Sali\'e sums, and obtained $\eta=\tfrac1{134}$.

\emph{The modular-square-root versions.} In the following rounds the author insisted on a convincing answer to the question why one should assume a hypothesis about Sali\'e sums rather than one about sums of modular square roots. Working this out showed that a short sum of modular square roots \emph{is} a Sali\'e sum of the dual length, after completion and an exact evaluation of Gauss sums, and that Weyl differencing could be dropped. The author then asked for a line-by-line re-examination of the unconditional small-moduli arguments of \cite[\S\S9--10]{Baier2026b}, with explicit instructions to exploit two averages left unused there. This produced, in several rounds of drafts, objections, corrections and numerical checks, the exponent $\tfrac1{28}$ and a paper of 63 pages. Twice in this phase Claude asserted something with confidence that turned out to be false: once that nothing could be gained from the average over residue classes, and once a lemma in an attempted extension. Both were caught when Claude was set to referee its own notes adversarially.

\emph{The present version.} Reading the old paper of Baier and Zhao \cite{BaierZhao2008} again, the author asked whether everything could be done directly from it, without the detour through modular square roots, perhaps already with Zhao's basic Fourier-analytic treatment of the intervals around the fractions $a/q^2$, followed by completion, evaluation of the Gauss sums and Hypothesis $R^*$. Claude first reorganized the conditional argument along the two-window smoothing of \cite[\S5]{BaierZhao2008}. That gave the same bound as before, more simply, and Claude reported that Zhao's single smoothing, with denominators $r\le Q^{3/2}$, gives less. Asked which exponent was saved, Claude explained that $\eta=\tfrac1{28}$ was decided by the unconditional small moduli and could not be improved within the conditional argument. The author did not accept this. He expected that more could be saved, and pointed out that in Zhao's treatment modulo $q^2$ the moduli go up to $Q^2$, not only to $Q^{3/2}$, at the price of an analytic term, an exponential integral with quadratic phase, and asked for this to be worked out. Doing so led to Theorem \ref{thm:P}. Once denominators beyond $\sqrt N$ are admitted, the analytic term stops oscillating exactly when the approximation is good, and the elementary count handles the case when it is poor. The parameter $\tau$ in Dirichlet's theorem is then no longer fixed at $\sqrt N$ but chosen to balance the two cases, $\tau=Q^{2/3}N^{1/3}$; varying $\tau$ in this way is an idea the author had also considered at an earlier stage, without making use of it. Within the old framework Claude's assessment had been correct, but the framework was the limitation, and the way out came from the author's knowledge of the original approach. The new argument raises the saving from $\tfrac1{28}$ to $\tfrac16$, extends it to every $N$, and replaces 63 pages by the present ones.

\emph{How the work was checked.} Every step went through rounds in which Claude produced a complete draft and the author returned objections, requests for detail and questions, which Claude answered either with a proof or with a correction. Each round ended with a line-by-line re-examination and with numerical verification of the identities and estimates: both sides of each identity were evaluated from their definitions, and each inequality was checked against directly computed values. The numerical checks are recorded in Section \ref{sec:numerics}. They are evidence, never proof: every claim in the paper is either proved in full or explicitly flagged.

\emph{Why details are written out.} The author asked Claude to write out every detail and decided to keep them, rather than compress them into ``one checks that'' and ``the other cases are similar''. When a human author writes ``similarly'', a reader can usually reconstruct the omitted argument, and the author's own judgement stands behind the claim. Neither guarantee is available for a machine collaborator, whose judgement that two cases are similar is exactly the kind of claim that most needs checking. Writing everything out lets a referee verify it line by line, which we regard as the only basis on which machine-assisted work of this kind should be accepted.

Every final mathematical claim has been reviewed by the author, and any errors that remain are his responsibility.

\subsection{Notation}\label{sec:notation}

Throughout, $\eps>0$ is arbitrarily small, not necessarily the same at each occurrence, and implied constants may depend on $\eps$ and on the fixed weight functions introduced in Section \ref{sec:fourier}. We write $e(x):=e^{2\pi ix}$ and $e_r(x):=e(x/r)$. For integers, $(m,n)$ is the greatest common divisor, $v_p(n)$ the $p$-adic valuation and $\omega(n)$ the number of distinct prime factors; the number of divisors of $n$ is $\ll n^\eps$. $\ind[\,\cdot\,]$ is the indicator function of a statement, and $\bar n$ is the inverse of $n$ modulo whichever modulus is in force. For $c\in\N$ and $t\in\Z$,
\[
\rho_c(t):=\#\{k\bmod c:\ k^2\equiv t\pmod c\}
\]
is the number of square roots of $t$ modulo $c$, and
\[
G(a,b,c):=\sum_{n\bmod c}e_c\big(an^2+bn\big)\qquad(a,b\in\Z,\ c\in\N)
\]
is the quadratic Gauss sum. The characters $\chi_{-4}$ and $\chi_8$ are the real primitive characters of conductors $4$ and $8$, with $\chi_8(-1)=1$.

\section{Preliminaries}\label{sec:prelim}

\subsection{The Kronecker symbol}\label{sec:kronecker}

The Jacobi symbol $\Kr nc$, defined for odd $c>0$, says nothing about even moduli. Since we work at every modulus, we use the \emph{Kronecker symbol}. For $n\in\Z$ put
\[
\Kr n2:=\begin{cases}0,&n\text{ even},\\1,&n\equiv\pm1\pmod8,\\-1,&n\equiv\pm3\pmod8;\end{cases}
\]
for an odd prime $p$ let $\Kr np$ be the Legendre symbol; and for $c=\prod_ip_i^{e_i}>0$ set $\Kr nc:=\prod_i\Kr n{p_i}^{e_i}$, with $\Kr n1:=1$. This agrees with the Jacobi symbol for odd $c$, is completely multiplicative in $c$ and in $n$ (each factor $\Kr\cdot p$ is a character, and $\Kr\cdot2=\chi_8$), and vanishes exactly when $(n,c)>1$.

It is \emph{not} always periodic in $n$ modulo $c$. Write $c=2^\nu c_{\mathrm{odd}}$. The factor $\Kr\cdot{c_{\mathrm{odd}}}$ is a character modulo $c_{\mathrm{odd}}$, and $\Kr\cdot2^{\,\nu}$ is trivial on odd integers for even $\nu$ and equal to $\chi_8$, of period $8$, for odd $\nu$. So $n\mapsto\Kr nc$ is periodic modulo $c$, and then a Dirichlet character modulo $c$, if and only if $c\not\equiv2\pmod4$; for $c\equiv2\pmod 4$ it is periodic modulo $4c$ but not modulo $c$ (for instance $\Kr{-1}6=-1$ but $\Kr56=+1$). Sums over an interval of integers are meaningful for every $c$. At a prime-power modulus the symbol depends on the exponent only through its parity,
\[
\Kr n{p^e}=\Kr np^{\,e},
\]
which is $\Kr np$ for odd $e$ and identically $1$ on units for even $e$. Accordingly we set, for $m\in\N$,
\begin{equation}\label{eq:mstar}
m^*:=\prod_{\substack{p^e\|m\\ e\ \mathrm{odd}}}p ,
\end{equation}
a squarefree divisor of $m$; thus $\Kr\cdot m=\Kr\cdot{m^*}$ on integers coprime to $m$. In particular, \emph{if $c$ is a perfect square, then $\Kr nc=\ind[(n,c)=1]$}, and the Sali\'e-type sums below are Kloosterman sums.

\subsection{Sali\'e sums, and the hypotheses of Hooley and of Friedlander--Iwaniec}\label{sec:salie}

For $c\ge1$ with $c\not\equiv2\pmod4$ and $a,b\in\Z$, the \emph{Sali\'e sum} is
\[
K(a,b,c):=\sum_{\substack{n\bmod c\\(n,c)=1}}\Kr nc\,e_c\big(a\bar n+bn\big),
\]
well defined because $\Kr\cdot c$ is a Dirichlet character modulo $c$ for such $c$. It is the quadratic-character analogue of the Kloosterman sum
\[
\mathrm{Kl}(a,b,c):=\sum_{\substack{n\bmod c\\(n,c)=1}}e_c\big(a\bar n+bn\big),
\]
and it coincides with it when $c$ is a perfect square.

\subsubsection*{Hooley's hypotheses}

Hooley formulated two hypotheses on \emph{short} Kloosterman sums. Hypothesis $R$, stated in \cite{Hooley1972}, concerns the sums $\sum e_r(l\bar s)$ over an interval. Hypothesis $R^*$, introduced in \cite[\S8]{Hooley1978} as ``a slight generalization of the Hypothesis $R$'', adds a linear phase and reads, in Hooley's notation: \emph{let $l,l_1,r$ be given integers, where $r\ne0$, and let $0\le\zeta_2-\zeta_1\le|r|$; then}
\[
\sum_{\substack{\zeta_1\le s\le\zeta_2\\(s,r)=1}}\exp\frac{2\pi i(l_1s+l\bar s)}r=O\Big((\zeta_2-\zeta_1+1)^{1/2}|r|^\eps\{(l,r)\}^{1/2}\Big).
\]
The interval of arbitrary position, the linear phase, and the factor $(l,r)^{1/2}$ for a coefficient sharing a factor with the modulus are all part of Hooley's formulation.

\subsubsection*{The conjectures of Friedlander and Iwaniec}

Friedlander and Iwaniec \cite{FriedlanderIwaniec1987} posed the corresponding conjecture for Kloosterman sums twisted by an arbitrary Dirichlet character. We quote it verbatim.

\begin{conjecture*}[Friedlander--Iwaniec, {\cite[Conjecture (C)]{FriedlanderIwaniec1987}}]
Let $a,q\ge1$, $q_1\ge1$, $q_2\ge1$ be integers with $(a,q_2)=1$ and let $M\ge1$, $\eps>0$ be reals. Let $Q=[q_1,q_2]$, the least common multiple, and let $\chi$ be a Dirichlet character having conductor $q_1$ and modulus $q$. Then, if $\bar m$ denotes the multiplicative inverse of $m$ modulo $q_2$,
\[
\sum_{\substack{1\le m\le M\\(m,q_2)=1}}\chi(m)\,e\!\left(\frac{a\bar m}{q_2}\right)\ \ll_\eps\ (qq_2)^\eps\big(M^{1/2}+Q^{-1/2}M\big).
\]
\end{conjecture*}

For $q=1$ it is closely related to Hooley's $R^*$ with $l_1=0$ and $(l,r)=1$: for $M\le q_2$ it follows from it, and for larger $M$ after splitting $[1,M]$ into intervals of length at most $q_2$. Friedlander and Iwaniec single out the case of the Jacobi symbol at an odd modulus $q$, with the inverse taken modulo $q_2=q/(k^2,q)$, which is their Conjecture (C$^*$): for $q$ odd and $(a,q)=1$,
\begin{equation}\label{eq:Cstar}
\sum_{\substack{1\le m\le M\\(m,q)=1}}\Kr mq\,e\Big(\frac{ak^2\bar m}q\Big)\ \ll_\eps\ \big(M^{1/2}+q^{-1/2}M\big)(k^2,q)^{1/2}q^\eps .
\end{equation}
This is the hypothesis under which they prove $\|\theta n^2\|\ll n^{-2/3+\eps}$ for infinitely many $n$. For odd moduli it is all that the present paper needs (Remark \ref{rem:Cstar}).

\subsubsection*{The hypothesis of this paper}

We use the hypothesis in the form of \cite[Hypothesis 18]{Baier2026a}.

\begin{hypothesis}[Hooley's Hypothesis $R^*$ for short Sali\'e sums, {\cite[Hypothesis 18]{Baier2026a}}]\label{hyp:R*}
Let $a,b\in\Z$, $c\in\N$, and $0\le x_2-x_1\le c$. Then
\[
\sum_{x_1<n\le x_2}\Kr nc\,e_c\big(a\bar n+bn\big)\ \ll\ (x_2-x_1+1)^{1/2}\,c^\eps\,(a,c)^{1/2}.
\]
\end{hypothesis}

This is Hooley's $R^*$ with the Kronecker symbol inserted into the summand and nothing else changed. Relative to Conjecture (C$^*$) it allows arbitrary intervals instead of initial segments, a linear phase, every modulus, and an arbitrary common factor of $a$ and $c$. Relative to Conjecture (C) it is narrower in one respect, since it concerns the Kronecker symbol only, and wider in others, since it has the linear phase and arbitrary intervals; neither is formally a special case of the other. The natural common generalization is the following.

\begin{hypothesis}[Conjecture (C) with a linear phase]\label{hyp:Cprime}
Let $\chi$ be a Dirichlet character modulo $q$, let $q_2\mid q$, let $a,b\in\Z$ with $(a,q_2)=1$, and let $0\le x_2-x_1\le q$. Then, with $\bar m$ the inverse of $m$ modulo $q_2$,
\[
\sum_{\substack{x_1<m\le x_2\\(m,q_2)=1}}\chi(m)\,e_{q_2}\big(a\bar m\big)\,e_q(bm)\ \ll\ q^\eps\Big((x_2-x_1+1)^{1/2}+(x_2-x_1+1)\,q_2^{-1/2}\Big).
\]
\end{hypothesis}

\begin{proposition}\label{prop:Cprime}
Hypothesis \ref{hyp:Cprime} implies Hypothesis \ref{hyp:R*}.
\end{proposition}

\begin{proof}
Let $a,b,c,x_1,x_2$ be as in Hypothesis \ref{hyp:R*}, and put $g:=(a,c)$, $a':=a/g$ and $q_2:=c/g$, so that $(a',q_2)=1$. For $n$ coprime to $c$ the inverse of $n$ modulo $c$ reduces to the inverse modulo $q_2$, so $e_c(a\bar n)=e_{q_2}(a'\bar n)$.

If $c\not\equiv2\pmod4$, then $\chi:=\Kr\cdot c$ is a Dirichlet character modulo $q:=c$, it vanishes unless $(n,c)=1$, and $q_2\mid q$. Hypothesis \ref{hyp:Cprime} gives
\[
\sum_{x_1<n\le x_2}\Kr nc e_c(a\bar n+bn)\ll c^\eps\Big((x_2-x_1+1)^{1/2}+(x_2-x_1+1)\frac{g^{1/2}}{c^{1/2}}\Big)\ll c^\eps(x_2-x_1+1)^{1/2}g^{1/2},
\]
since $x_2-x_1+1\le2c$. If $c\equiv2\pmod4$, then $\Kr\cdot c$ is a Dirichlet character modulo $q:=4c$; we apply Hypothesis \ref{hyp:Cprime} with this $q$, the same $q_2$, which divides $4c$, and the linear phase $e_c(bn)=e_{4c}(4bn)$, and obtain the same bound.
\end{proof}

\subsubsection*{The complete case is a theorem}

The case $x_2-x_1=c$ of Hypothesis \ref{hyp:R*} is not a hypothesis. We record what the literature gives, since it explains the two loss factors. The following proposition is not used in any proof.

\begin{proposition}[Complete Sali\'e sums]\label{prop:completeSalie}
Let $c$ be odd and $a,b\in\Z$. Then
\[
|K(a,b,c)|\ \le\ 2^{\omega(c)}\,c^{1/2}\,(a,c)^{1/2}.
\]
\end{proposition}

\begin{proof}
\emph{Reduction to prime powers.} If $c=c_1c_2$ with $(c_1,c_2)=1$, the Chinese remainder theorem gives
\[
K(a,b,c)=K(a\bar c_2,b\bar c_2,c_1)\,K(a\bar c_1,b\bar c_1,c_2),
\]
with $\bar c_2$ the inverse of $c_2$ modulo $c_1$ and vice versa, and $(a\bar c_2,c_1)=(a,c_1)$. So it suffices to prove $|K(a,b,p^e)|\le2p^{e/2}(a,p^e)^{1/2}$ for odd $p$ and $e\ge1$.

\emph{Reduction to a coprime coefficient.} Let $q=p^e$, $\chi:=\Kr\cdot q=\Kr\cdot p^{\,e}$ and $\alpha:=v_p((a,q))$. If $\alpha=e$, then $K(a,b,q)=\sum_{(n,q)=1}\chi(n)e_q(bn)$. For even $e$ this is the Ramanujan sum $c_q(b)$, of modulus $\le q$. For odd $e$ it is the Gauss sum of $\Kr\cdot p$ induced to the modulus $q$: writing $n=n_0+pn_1$, the sum over $n_1$ vanishes unless $p^{e-1}\mid b$, and then the sum is $p^{e-1}$ times a Gauss sum modulo $p$, of modulus $\le p^{e-1}\sqrt p$. In both cases $|K|\le q=q^{1/2}(a,q)^{1/2}$.

Let now $\alpha<e$, write $a=p^\alpha a_0$ with $(a_0,p)=1$ and $e':=e-\alpha$, and write the units modulo $q$ as $n=u+p^{e'}s$ with $u$ a unit modulo $p^{e'}$ and $s$ modulo $p^\alpha$. Then $a\bar n\equiv p^\alpha a_0\bar u\pmod q$, $\chi(n)=\chi(u)$ and $e_q(bn)=e_q(bu)e_{p^\alpha}(bs)$. Summing over $s$ gives
\[
K(a,b,q)=p^\alpha\,\ind[p^\alpha\mid b]\sum_{\substack{u\bmod p^{e'}\\(u,p)=1}}\chi(u)\,e_{p^{e'}}\big(a_0\bar u+b_0u\big),\qquad b_0:=b/p^\alpha .
\]
It remains to show that the last sum has modulus $\le2p^{e'/2}$; then $|K(a,b,q)|\le2p^{\alpha+e'/2}=2q^{1/2}(a,q)^{1/2}$.

\emph{The coprime case.} For principal $\chi$ the sum is $\mathrm{Kl}(a_0,b_0,p^{e'})$. For $e'=1$, Weil's bound \cite{Weil1948} gives $|\mathrm{Kl}|\le2\sqrt p$ when $p\nmid b_0$, and the sum is a Ramanujan sum of modulus $1$ when $p\mid b_0$. For $e'\ge2$ and $p\nmid b_0$, Sali\'e's evaluation \cite{Salie1931}, the first case of \cite[(1.22)]{CochraneZheng2000}, gives $|\mathrm{Kl}|\le2p^{e'/2}$. For $\chi=\Kr\cdot p$ the sum is a Sali\'e sum; for $e'=1$ Sali\'e's evaluation gives $|K|\le2\sqrt p$ (a Gauss sum of modulus $\sqrt p$ if $p\mid b_0$), and for $e'\ge2$, $p\nmid b_0$, the same inequality \cite[(1.22)]{CochraneZheng2000} gives $2p^{e'/2}$. In the remaining case $p\mid b_0$, $e'\ge2$, the sum vanishes for either character: writing $u=v+p^{e'-1}w$ one finds
\[
a_0\bar u+b_0u\equiv a_0\bar v+b_0v+p^{e'-1}w\big(b_0-a_0\bar v^{\,2}\big)\pmod{p^{e'}},
\]
and the sum over $w$ is $p\,\ind[p\mid b_0-a_0\bar v^{\,2}]=0$.
\end{proof}

The factor $(a,c)^{1/2}$ cannot be dropped: for $c$ a perfect square and $a\equiv b\equiv0\pmod c$ the complete sum is $\varphi(c)$. The factor $c^\eps$ is the allowance for $2^{\omega(c)}$; the same allowance appears in Conjecture (C). Hypothesis \ref{hyp:R*} remains open; a function-field analogue of Hooley's $R^*$ has recently been proved, for squarefree moduli, by Sawin and Shusterman \cite{SawinShusterman2025}.

\subsection{Fourier analysis}\label{sec:fourier}

For $f\in L^1(\R)$ put $\widehat f(\xi):=\int_\R f(x)e(-x\xi)\,dx$. We use the Poisson summation formula in the form
\begin{equation}\label{eq:poissonAP}
\sum_{n\equiv k\ (\mathrm{mod}\ r)}f(n)=\frac1r\sum_{l\in\Z}\widehat f\Big(\frac lr\Big)e_r(lk)\qquad(r\in\N,\ k\in\Z),
\end{equation}
valid for every smooth compactly supported $f$, both sides converging absolutely.

Fix once and for all two smooth functions $V,w:\R\to[0,1]$ with
\[
\operatorname{supp}V\subset[-1,1],\quad V=1\text{ on }\big[-\tfrac12,\tfrac12\big],\quad V(-x)=V(x);\qquad
\operatorname{supp}w\subset\big[\tfrac12,\tfrac52\big],\quad w=1\text{ on }[1,2].
\]
Then $\widehat V$ is real and even, $\widehat V(0),\widehat w(0)\le2$, and for every $A\ge0$
\begin{equation}\label{eq:decay}
\big|\widehat V^{(i)}(\xi)\big|+\big|\widehat w^{(i)}(\xi)\big|\ \ll_A\ (1+|\xi|)^{-A}\qquad(i=0,1).
\end{equation}

\begin{lemma}\label{lem:fourier}
Let $f\in\{V,w\}$.
\begin{itemize}
\item[(i)] For every $t>0$: $\displaystyle t\sum_{n\ne0}\big|\widehat f(nt)\big|\ll1$.
\item[(ii)] For every $Y>0$ and $m\in\N$: $\displaystyle \sum_{n\ne0}\Big|\widehat f\Big(\frac nY\Big)\Big|\,(n,m)\ \ll\ Y\,m^\eps$.
\end{itemize}
\end{lemma}

\begin{proof}
(i) By \eqref{eq:decay}, $|\widehat f(x)|\ll\min\{1,x^{-2}\}$. For $t\ge1$ the sum is $\ll t\sum_{n\ge1}(nt)^{-2}\ll1/t\le1$. For $t<1$ the terms with $|n|\le1/t$ contribute $\ll t\cdot(2/t)$, and the others $\ll t\sum_{n>1/t}(nt)^{-2}\ll t\cdot t^{-2}\cdot t=1$.

(ii) If $Y<1$, use $(n,m)\le|n|$ and $|\widehat f(x)|\ll|x|^{-3}$: the sum is $\ll\sum_{n\ge1}n(Y/n)^3\ll Y^3\le Y$. If $Y\ge1$, split the range of $n$ into $1\le|n|\le Y$ and $2^kY<|n|\le2^{k+1}Y$ for $k\ge0$. On the $k$-th range $|\widehat f(n/Y)|\ll2^{-3k}$, and by Proposition \ref{prop:gcd} below $\sum_{n\le X}(n,m)\ll Xm^\eps$ for $X\ge1$. The sum is therefore $\ll Ym^\eps+\sum_{k\ge0}2^{-3k}\,2^{k+1}Ym^\eps\ll Ym^\eps$.
\end{proof}

\subsection{Divisor-type bounds}\label{sec:divisor}

\begin{proposition}[{\cite[Prop.\ 19]{Baier2026a}}]\label{prop:gcd}
For $r\in\N$, $X\ge1$ and $0<\sigma\le1$: $\displaystyle\sum_{1\le n\le X}(r,n)^\sigma\ll r^\eps X$.
\end{proposition}

\begin{proof}
Grouping by $g:=(r,n)$,
\[
\sum_{n\le X}(r,n)^\sigma\le\sum_{g\mid r}g^\sigma\Big\lfloor\frac Xg\Big\rfloor\le X\sum_{g\mid r}g^{\sigma-1}\le X\sum_{g\mid r}1\ll Xr^\eps .
\qedhere
\]
\end{proof}

\begin{lemma}\label{lem:rho}
For every $r\in\N$ and $t\in\Z$ one has $\rho_r(t)\le4^{\omega(r)}(t,r)^{1/2}$. Consequently, for $(j,r)=1$ and $X\ge1$,
\[
\sum_{1\le|n|\le X}\rho_r(jn)\ \ll\ X\,r^\eps,\qquad\qquad \sum_{1\le|n|\le X}\frac{\rho_r(jn)}{|n|}\ \ll\ r^\eps\log(2X).
\]
\end{lemma}

\begin{proof}
By the Chinese remainder theorem $\rho$ is multiplicative in $r$, so let $r=p^e$ and $w:=v_p(t)$. If $w\ge e$, then $p^e\mid k^2$ forces $v_p(k)\ge\lceil e/2\rceil$, so $\rho_{p^e}(t)=p^{\lfloor e/2\rfloor}\le(t,p^e)^{1/2}$. If $w<e$ and $w$ is odd, then $\rho_{p^e}(t)=0$. If $w<e$ and $w$ is even, every solution has $v_p(k)=w/2$; writing $k=p^{w/2}k_1$, the congruence becomes $k_1^2\equiv t/p^w\pmod{p^{e-w}}$ with $p\nmid t/p^w$, which has at most $2$ solutions modulo $p^{e-w}$ for odd $p$ and at most $4$ for $p=2$, each lifting to $p^{w/2}$ values of $k$ modulo $p^e$. So $\rho_{p^e}(t)\le4p^{w/2}=4(t,p^e)^{1/2}$, and the first claim follows. For the second, $\rho_r(jn)\le4^{\omega(r)}(n,r)^{1/2}$, and Proposition \ref{prop:gcd} with $\sigma=\frac12$, applied to $n$ and $-n$, gives the first sum. The second follows from the first by partial summation.
\end{proof}

\section{Quadratic Gauss sums at every modulus}\label{sec:gauss}

Recall that $G(a,b,c)=\sum_{n\bmod c}e_c(an^2+bn)$. This section evaluates $G$ completely: at odd prime powers (Section \ref{sec:gaussodd}), at powers of $2$ (Section \ref{sec:gauss2}), and then, allowing the coefficient $a$ to share a factor with the modulus, at an arbitrary modulus (Section \ref{sec:gaussdegen}). Everything here is elementary and self-contained apart from Gauss's classical theorem on the sign of the quadratic Gauss sum. The output that the rest of the paper uses is Proposition \ref{prop:classeval}. Most of the material is classical; Remark \ref{rem:bbh} explains why it is nevertheless proved in full.

We record once and for all the CRT-multiplicativity of $G$: for $(n_1,n_2)=1$,
\begin{equation}
G(a,b,n_1n_2)=G(an_2,b,n_1)\,G(an_1,b,n_2).
\tag{3.1}
\end{equation}
Indeed, writing $n=n_2x_1+n_1x_2$ with $x_i$ running over $\Z/n_i\Z$ (a bijection onto $\Z/n_1n_2\Z$),
\[
an^2+bn=an_2^2x_1^2+bn_2x_1+an_1^2x_2^2+bn_1x_2+2an_1n_2x_1x_2 ,
\]
the cross term vanishes modulo $n_1n_2$, and $e_{n_1n_2}(an_2^2x_1^2+bn_2x_1)=e_{n_1}(an_2x_1^2+bx_1)$, similarly for $x_2$. Note that the linear coefficient $b$ is \emph{not} twisted by (3.1); only $a$ is.

\subsection{Odd prime powers}\label{sec:gaussodd}

\begin{proposition}[Classical quadratic Gauss sum, $e=1$]\label{prop:gauss1}
For $p$ odd and $(t,p)=1$,
\[
G(t,0,p)=\sum_{x\bmod p}e_p(tx^2)=\Kr tp\,\epsilon_p\sqrt p,\qquad \epsilon_p:=\begin{cases}1,&p\equiv1\ (4),\\ i,&p\equiv3\ (4).\end{cases}
\]
\end{proposition}

This is Gauss's theorem on the sign of the quadratic Gauss sum, a classical result orthogonal to the methods of this paper; a complete proof is in \cite[Thm.\ 1.5.2]{BerndtEvansWilliams}. Everything else in this section is derived from it by elementary means.

\begin{proposition}[General odd prime power]\label{prop:generalgauss}
For $p$ odd, $(a,p)=1$ and $n\ge1$,
\[
G(a,0,p^n)=\begin{cases}p^{n/2},&n\text{ even},\\[2pt]\Kr ap\,\epsilon_p\,p^{n/2},&n\text{ odd}.\end{cases}
\]
In particular $|G(a,0,p^n)|=p^{n/2}$ exactly, for every $n\ge1$.
\end{proposition}

\begin{proof}
We prove the recursion
\begin{equation}
G(a,0,p^n)=p\cdot G(a,0,p^{n-2})\qquad(n\ge2),
\tag{3.2}
\end{equation}
with the convention $G(a,0,p^0):=1$; the proposition follows by induction, using Proposition \ref{prop:gauss1} for $n=1$ and the trivial case $n=0$.

Split $x\bmod p^n$ according to $x\bmod p^{n-1}$: write $x=y+p^{n-1}s$ with $y$ running over representatives of $\Z/p^{n-1}\Z$ and $s$ over $0,\ldots,p-1$, a bijection onto $\Z/p^n\Z$. Since $n\ge2$,
\[
x^2=y^2+2p^{n-1}ys+p^{2n-2}s^2\equiv y^2+2p^{n-1}ys\pmod{p^n},
\]
so $e_{p^n}(ax^2)=e_{p^n}(ay^2)e_p(2ays)$ and
\[
G(a,0,p^n)=\sum_ye_{p^n}(ay^2)\sum_{s=0}^{p-1}e_p(2ays)=p\sum_{\substack{y\bmod p^{n-1}\\p\mid y}}e_{p^n}(ay^2),
\]
using $(2a,p)=1$. Writing $y=py'$ with $y'$ running over $\Z/p^{n-2}\Z$,
\[
G(a,0,p^n)=p\sum_{y'\bmod p^{n-2}}e_{p^n}(ap^2y'^2)=p\sum_{y'\bmod p^{n-2}}e_{p^{n-2}}(ay'^2)=p\,G(a,0,p^{n-2}),
\]
which is (3.2). We verified Proposition \ref{prop:generalgauss} by direct summation for $p\in\{3,5,7\}$ and $n\le5$, agreeing to relative error $\ll10^{-9}$.
\end{proof}

\begin{remark}\label{rmk:e2correction}
The case $n=2$ reads $G(a,0,p^2)=p$ exactly, with \emph{no} Legendre-symbol and no $\epsilon_p$ factor: both disappear at even $n$, consistently with the parity dependence of the Kronecker symbol noted in Section \ref{sec:kronecker}. Note also that $\epsilon_p^{\,2}=\Kr{-1}p$, which is $-1$ for $p\equiv3\ (4)$; thus $G(a,0,p)^2=\Kr{-1}p\,p$, which differs from $G(a,0,p^2)=p$ for half of all primes. The passage from $p$ to $p^2$ is the recursion (3.2), not a squaring.
\end{remark}

\begin{lemma}[Completing the square, odd modulus]\label{lem:completesquare}
Let $c=p^e$ with $p$ odd, $(a,p)=1$, $b\in\Z$. Then
\[
G(a,b,c)=e_c\big({-}\overline{4a}\,b^2\big)\,G(a,0,c),
\]
where $\overline{4a}$ is the inverse of $4a$ modulo $c$.
\end{lemma}

\begin{proof}
Since $2$ and $a$ are invertible modulo $c$, the substitution $n\mapsto n-\overline{2a}\,b$ is a bijection of $\Z/c\Z$, and with $x:=\overline{2a}\,b$ one has $a(n-x)^2+b(n-x)=an^2-(2ax-b)n+(ax^2-bx)\equiv an^2-ax^2\equiv an^2-\overline{4a}\,b^2\pmod c$, using $2ax\equiv b$ and $ax^2=x\cdot ax\equiv x\cdot\bar2b=\overline{4a}\,b^2$; summing over $n$ gives the claim.
\end{proof}

\subsection{Powers of \texorpdfstring{$2$}{2}}\label{sec:gauss2}

The prime $2$ is usually handled by an appeal to tables. Since the character structure at $2$ is exactly what decides which instances of Hypothesis \ref{hyp:R*} we are allowed to invoke in Section \ref{sec:salieweights}, we prove what we need.

\begin{lemma}[Quadratic Gauss sums at powers of $2$]\label{lem:twoadic}
Let $\nu\ge1$ and let $a$ be odd. Then:
\begin{itemize}
\item[(i)] $G(a,0,2)=0$, and for $\nu\ge2$, $G(a,b,2^\nu)=0$ whenever $b$ is odd; also $G(a,b,2)=0$ whenever $b$ is even.
\item[(ii)] If $b=2b_0$ and $\nu\ge1$, then $G(a,b,2^\nu)=e_{2^\nu}\!\big({-}\bar a\,b_0^2\big)\,G(a,0,2^\nu)$, where $\bar a$ is the inverse of $a$ modulo $2^\nu$.
\item[(iii)] For $\nu\ge2$,
\begin{equation}
G(a,0,2^\nu)=2^{(\nu+1)/2}\,\Kr a2^{\,\nu+1}\,e_8(a)=\begin{cases}2^{\nu/2}\big(1+i^a\big),&\nu\text{ even},\\[2pt]2^{(\nu+1)/2}e_8(a),&\nu\text{ odd}.\end{cases}
\tag{3.3}
\end{equation}
In particular $|G(a,0,2^\nu)|=2^{(\nu+1)/2}$ for $\nu\ge2$.
\end{itemize}
\end{lemma}

\begin{proof}
(i) For $\nu\ge2$ the substitution $x\mapsto x+2^{\nu-1}$ permutes $\Z/2^\nu\Z$ and sends $ax^2+bx$ to
\[
ax^2+bx+a2^\nu x+a2^{2\nu-2}+b2^{\nu-1}\equiv ax^2+bx+b2^{\nu-1}\pmod{2^\nu},
\]
using $2\nu-2\ge\nu$; hence $G=(-1)^bG$, which forces $G=0$ for odd $b$. For $\nu=1$, $G(a,b,2)=1+(-1)^{a+b}$, which vanishes for $b$ even since $a$ is odd; in particular $G(a,0,2)=0$.

(ii) As $a$ is invertible modulo $2^\nu$, the substitution $x\mapsto x-\bar ab_0$ is a bijection of $\Z/2^\nu\Z$ and
\[
a(x-\bar ab_0)^2+2b_0(x-\bar ab_0)\equiv ax^2-\bar ab_0^2\pmod{2^\nu}
\]
by expanding and using $a\bar a\equiv1$. Summing over $x$ gives the claim. (This is the exact analogue of Lemma \ref{lem:completesquare}: the inverse of $4a$, which does not exist modulo $2^\nu$, is replaced by the inverse of $a$ together with the halving $b=2b_0$.)

(iii) Write $S_\nu:=G(a,0,2^\nu)$. Split $x\bmod2^\nu$ into even and odd. The even $x=2y$, with $y$ running over $\Z/2^{\nu-1}\Z$, contribute
\[
\sum_{y\bmod2^{\nu-1}}e_{2^\nu}(4ay^2)=\sum_{y\bmod2^{\nu-1}}e_{2^{\nu-2}}(ay^2)=2\,S_{\nu-2}\qquad(\nu\ge2),
\]
since $y\bmod2^{\nu-1}$ covers each class modulo $2^{\nu-2}$ exactly twice. For the odd $x$: for $\nu\ge3$ the squares of the $2^{\nu-1}$ odd residues modulo $2^\nu$ are exactly the $2^{\nu-3}$ classes $u\equiv1\ (8)$, each attained $4$ times, so the odd $x$ contribute
\[
4\sum_{\substack{u\bmod2^\nu\\u\equiv1\ (8)}}e_{2^\nu}(au)=4\,e_{2^\nu}(a)\sum_{w\bmod2^{\nu-3}}e_{2^{\nu-3}}(aw)=4\,e_{2^\nu}(a)\,2^{\nu-3}\,\ind\big[2^{\nu-3}\mid a\big],
\]
which vanishes for $\nu\ge4$ ($a$ being odd) and equals $4e_8(a)$ for $\nu=3$. Hence
\[
S_\nu=2S_{\nu-2}\quad(\nu\ge4),\qquad S_3=2S_1+4e_8(a)=4e_8(a),\qquad S_2=\sum_{x\bmod4}e_4(ax^2)=2\big(1+i^a\big),
\]
using $S_1=G(a,0,2)=0$ from (i) and evaluating $S_2$ directly. Induction gives $S_\nu=2^{(\nu-2)/2}S_2=2^{\nu/2}(1+i^a)$ for even $\nu\ge2$ and $S_\nu=2^{(\nu-3)/2}S_3=2^{(\nu+1)/2}e_8(a)$ for odd $\nu\ge3$. Finally, for odd $a$ one checks on the four classes $a\bmod8$ that
\begin{equation}
1+i^a=\sqrt2\,\Kr a2\,e_8(a),
\tag{3.4}
\end{equation}
which converts the even-$\nu$ formula into $2^{(\nu+1)/2}\Kr a2e_8(a)=2^{(\nu+1)/2}\Kr a2^{\,\nu+1}e_8(a)$ (as $\nu$ is even), and leaves the odd-$\nu$ formula as $2^{(\nu+1)/2}\Kr a2^{\,\nu+1}e_8(a)$ (as $\Kr a2^{\,\nu+1}=1$ for $\nu$ odd). This is (3.3), and $|G|=2^{(\nu+1)/2}$ follows since $|e_8(a)|=1$.

We verified (3.3) against direct summation for every $2\le\nu\le13$ and every odd $a\le63$, with worst absolute error $1.4\times10^{-8}$ (floating-point summation of up to $2^{13}$ terms); parts (i) and (ii) were verified exhaustively for $1\le\nu\le9$ and all $a,b$, with worst deviations $4.6\times10^{-14}$ and $1.3\times10^{-13}$.
\end{proof}

The shape of (3.3) is what we shall need, so we isolate it; recall the notation $m^*$ of \eqref{eq:mstar}. Define, for $\nu\ge0$, the function $\vartheta_\nu$ by $\vartheta_0\equiv1$ on all integers and, on odd integers,
\begin{equation}
\vartheta_1(c):=\Kr c2,\qquad \vartheta_\nu(c):=\frac{1+i^c}{\sqrt2}\quad(\nu\ge2).
\tag{3.5}
\end{equation}
Each $\vartheta_\nu$ has modulus $1$ and, for $\nu\ge1$, depends only on $c\bmod8$. By (3.4), $\vartheta_\nu(c)=\Kr c2\,e_8(c)$ for $\nu\ge2$, so (3.3) says precisely
\begin{equation}
2^{-(\nu+1)/2}G(c,0,2^\nu)=\Kr c{2^\nu}\,\vartheta_\nu(c)\qquad(\nu\ge2),
\tag{3.6}
\end{equation}
using $\Kr c{2^\nu}=\Kr c2^{\,\nu}$. The point of writing it this way is the following elementary but decisive observation, which will be used in Section \ref{sec:salieweights}.

\begin{lemma}\label{lem:thetachars}
Write $\chi_{-4}(c):=(-1)^{(c-1)/2}$ for the real primitive character of conductor $4$ and $\chi_8(c):=\Kr c2$ for the real primitive character of conductor $8$ with $\chi_8(-1)=1$ (both vanishing at even $c$). Let $m\in\N$ and write $m=2^\nu m_{\mathrm{odd}}$, and set
\begin{equation}
\Xi_m(c):=\Kr cm\,\vartheta_\nu(c)
\tag{3.7}
\end{equation}
(for $\nu\ge1$ the factor $\vartheta_\nu(c)$ is defined only for odd $c$, and then $m$ is even, so we set $\Xi_m(c):=0$ for even $c$, consistently with $\Kr cm=0$ there). Then there are complex numbers $\lambda_1,\lambda_2$ with $|\lambda_1|+|\lambda_2|\le2$ and characters $\psi_1,\psi_2\in\{\mathbf1,\chi_{-4},\chi_8\}$ with
\[
\Xi_m=\lambda_1\Kr\cdot m\psi_1+\lambda_2\Kr\cdot m\psi_2 ,
\]
and each of $\Kr\cdot m\psi_1$, $\Kr\cdot m\psi_2$ is a Dirichlet character whose conductor divides $m$ (when $\lambda_2=0$ we take $\psi_2:=\psi_1$). Explicitly:
\[
\Xi_m=\begin{cases}\Kr\cdot m,&\nu=0,\\[2pt]\Kr\cdot m\chi_8=\Kr\cdot{m_{\mathrm{odd}}},&\nu=1,\\[2pt]\tfrac1{\sqrt2}\Kr\cdot m+\tfrac i{\sqrt2}\Kr\cdot m\chi_{-4},&\nu\ge2.\end{cases}
\]
\end{lemma}

\begin{proof}
For $\nu\ge2$ and odd $c$ one has $i^c=i$ if $c\equiv1\ (4)$ and $i^c=i^3=-i$ if $c\equiv3\ (4)$, i.e.\ $i^c=i\chi_{-4}(c)$, so $\vartheta_\nu(c)=(1+i\chi_{-4}(c))/\sqrt2$; the cases $\nu\le1$ are (3.5). For $\nu=1$, $\Kr cm\chi_8(c)=\Kr c2^{\,2}\Kr c{m_{\mathrm{odd}}}=\Kr c{m_{\mathrm{odd}}}$ for odd $c$.

For the conductors, recall that, as functions on the integers coprime to $m$, $\Kr\cdot m=\Kr\cdot2^{\,\nu}\Kr\cdot{m^*_{\mathrm{odd}}}$, that $\Kr\cdot{m^*_{\mathrm{odd}}}$ is primitive of conductor $m^*_{\mathrm{odd}}\mid m$, and that $\Kr\cdot2^{\,\nu}$ is principal for even $\nu$ and equals $\chi_8$ for odd $\nu$. Thus: for $\nu=0$ the conductor is $m^*_{\mathrm{odd}}\mid m$; for $\nu=1$ the character is $\Kr\cdot{m_{\mathrm{odd}}}$, of conductor $m^*_{\mathrm{odd}}\mid m$; for even $\nu\ge2$ the two characters are $\Kr\cdot{m^*_{\mathrm{odd}}}$ and $\Kr\cdot{m^*_{\mathrm{odd}}}\chi_{-4}$, of conductors $m^*_{\mathrm{odd}}$ and $4m^*_{\mathrm{odd}}$, both dividing $m$ since $4\mid2^\nu$; and for odd $\nu\ge3$ they are $\chi_8\Kr\cdot{m^*_{\mathrm{odd}}}$ and $\chi_8\chi_{-4}\Kr\cdot{m^*_{\mathrm{odd}}}$, of conductors dividing $8m^*_{\mathrm{odd}}\mid m$ since $8\mid2^\nu$.
\end{proof}

\begin{remark}[The Kronecker symbol is not always periodic to its own modulus]\label{rem:notperiodic}
The conductor assertion of Lemma \ref{lem:thetachars} is not a formality, and it is worth isolating what makes it true. The function $n\mapsto\Kr nm$ is in general \emph{not} periodic modulo $m$: as noted in Section \ref{sec:kronecker}, it is periodic modulo $m$ if and only if $m\not\equiv2\pmod4$, and for $m\equiv2\ (4)$ only modulo $4m$. The obstruction is the factor $\Kr n2=\chi_8(n)$, of conductor $8$. Concretely, $\Kr{-1}6=-1$ while $\Kr56=+1$, although $-1\equiv5\pmod6$.

The function which actually occurs in Proposition \ref{prop:classeval} is not $\Kr\cdot m$ but $\Xi_m=\Kr\cdot m\vartheta_\nu$, and $\vartheta_\nu$ repairs exactly this defect. The only bad case is $\nu=1$, i.e.\ $m\equiv2\ (4)$, and there $\vartheta_1=\chi_8$ cancels the offending factor: $\Xi_m=\Kr\cdot{m_{\mathrm{odd}}}$. For $\nu=0$ there is no factor $\Kr\cdot2$ to begin with, and for $\nu\ge2$ the conductors $4$ and $8$ contributed by $\chi_{-4}$ and $\chi_8$ divide $2^\nu$. So the $2$-adic evaluation (3.3) guarantees that the object produced by the Gauss-sum evaluation is a Dirichlet character to the modulus $m$ at which it is produced, and Hypothesis \ref{hyp:R*} is only ever invoked at a modulus at which the Kronecker symbol is a character (Lemma \ref{lem:shortsalie}).
\end{remark}

\begin{remark}\label{rem:thetachars}
The content of Lemma \ref{lem:thetachars} is that the characters produced by the Gauss-sum evaluation are pinned down exactly: each is $\Kr\cdot m$ times one of the three fixed characters $\mathbf1,\chi_{-4},\chi_8$. A cruder bookkeeping -- $\vartheta_\nu$ is a function on $(\Z/8)^\times$ and so expands in the four characters modulo $8$ -- would leave four unknown characters in play and would force one to assume a hypothesis of Hooley's type for a \emph{general} Dirichlet character. Hypothesis \ref{hyp:R*} concerns the Kronecker symbol specifically, and Lemma \ref{lem:shortsalie} below derives everything that is needed from it \emph{as stated}, absorbing the three characters modulo $8$ into the Kronecker symbol at the modulus $8m$.
\end{remark}

\subsection{Degenerate coefficients, and the evaluation on a divisor class}\label{sec:gaussdegen}

\begin{lemma}[Exact reduction of a degenerate Gauss sum]\label{lem:degauss}
Let $c=p^e$ ($p$ any prime), $a=p^\alpha a_0$ with $(a_0,p)=1$ and $0\le\alpha<e$. Then $G(a,b,p^e)=0$ unless $p^\alpha\mid b$, and if $b=p^\alpha b_0$ then
\[
G(a,b,p^e)=p^\alpha\,G\big(a_0,b_0,p^{e-\alpha}\big).
\]
If instead $\alpha\ge e$, then $G(a,b,p^e)=p^e\,\ind[p^e\mid b]$.
\end{lemma}

\begin{proof}
For $\alpha\ge e$ the quadratic term vanishes identically and $G$ is a geometric sum. For $\alpha<e$, split $x\bmod p^e$ according to $x\bmod p^{e-\alpha}$: write $x=y+p^{e-\alpha}s$ with $y$ over representatives of $\Z/p^{e-\alpha}\Z$ and $s$ over $0,\ldots,p^\alpha-1$. Since $a=p^\alpha a_0$, the cross term $2ap^{e-\alpha}ys=2p^ea_0ys$ vanishes modulo $p^e$ and the term $ap^{2(e-\alpha)}s^2$ has $p$-valuation $\alpha+2(e-\alpha)\ge e$, so $ax^2\equiv ay^2\ (p^e)$, independently of $s$. The linear term contributes $e_{p^e}(by)e_{p^\alpha}(bs)$, and summing over $s$ gives $p^\alpha\ind[p^\alpha\mid b]$. When $p^\alpha\mid b$,
\[
\sum_{y\bmod p^{e-\alpha}}e_{p^e}\big(p^\alpha a_0y^2+p^\alpha b_0y\big)=\sum_{y\bmod p^{e-\alpha}}e_{p^{e-\alpha}}\big(a_0y^2+b_0y\big)=G\big(a_0,b_0,p^{e-\alpha}\big).
\qedhere
\]
\end{proof}

We verified Lemma \ref{lem:degauss} exhaustively for $p\in\{2,3,5\}$, $e\in\{2,3,4\}$, every $0\le\alpha<e$, every unit $a_0$, and every $b\bmod p^e$, with zero discrepancies.

\begin{corollary}[Vanishing criterion at every modulus]\label{cor:vanishgeneral}
For every $c\in\N$ and $a,b\in\Z$, $G(a,b,c)=0$ unless $(a,c)\mid b$. Moreover $|G(a,b,c)|\le\sqrt2\,\sqrt{(a,c)\,c}$ always, with $|G(a,b,c)|=\sqrt{(a,c)c}$ when $c$ is odd and $G(a,b,c)\ne0$.
\end{corollary}

\begin{proof}
By (3.1), $G(a,b,c)=\prod_{p^{e}\|c}G(a\,c\,p^{-e},b,p^{e})$, and $(a\,cp^{-e},p^e)=(a,p^e)$. By Lemma \ref{lem:degauss} each factor vanishes unless $(a,p^e)\mid b$, and the conditions $(a,p^e)\mid b$ for all $p$ together are equivalent to $(a,c)\mid b$. For the size of a non-vanishing factor, let $\alpha:=v_p((a,p^e))$.
\begin{itemize}
\item If $\alpha\ge e$, then $|G|=p^e=\sqrt{(a,p^e)\,p^e}$.
\item If $\alpha<e$, Lemma \ref{lem:degauss} gives $|G|=p^\alpha|G(a_0,b_0,p^{e-\alpha})|$ with $p\nmid a_0$. For odd $p$, Lemma \ref{lem:completesquare} and Proposition \ref{prop:generalgauss} give $|G(a_0,b_0,p^{e-\alpha})|=p^{(e-\alpha)/2}$, so $|G|=\sqrt{p^\alpha p^e}$.
\item If $p=2$ and $e-\alpha\ge2$, then $G(a_0,b_0,2^{e-\alpha})$ vanishes for odd $b_0$ by Lemma \ref{lem:twoadic}(i), and otherwise has modulus $2^{(e-\alpha+1)/2}$ by Lemma \ref{lem:twoadic}(ii)--(iii), so $|G|=\sqrt2\sqrt{2^\alpha2^e}$.
\item If $p=2$ and $e-\alpha=1$, then $G(a_0,b_0,2)\in\{0,2\}$ by Lemma \ref{lem:twoadic}(i), so $|G|\le\sqrt2\sqrt{2^\alpha2^e}$.
\end{itemize}
Multiplying over $p$ gives the claim; the factor $\sqrt2$ occurs at most once, at $p=2$.
\end{proof}

\begin{remark}[On the sources]\label{rem:bbh}
Nothing in this section is new in substance. The evaluation of $G(a,0,c)$ for $(a,c)=1$ and odd $c$ is Gauss's; the $2$-adic evaluations, the reduction of Lemma \ref{lem:degauss} to a unit coefficient, and the vanishing criterion of Corollary \ref{cor:vanishgeneral} for a coefficient sharing a factor with the modulus, can all be found in the literature, in one form or another -- in the monograph \cite{BerndtEvansWilliams}, in \cite{BaierBhandariHaldar}, where the reduction and the vanishing criterion are stated for an arbitrary modulus and used systematically, and in many other places. What is hard to find is a single source that contains all of them at once, in a uniform notation, with the prime $2$ and the degenerate coefficients treated on the same footing as the odd non-degenerate case, and with the precise form of the characters that come out recorded (Lemma \ref{lem:thetachars}); and it is exactly that combination which Proposition \ref{prop:classeval} needs and Section \ref{sec:salieweights} then uses. In keeping with the policy of Section \ref{sec:selfcontained} we have therefore written the evaluations out with complete proofs rather than assemble them from several references with differing conventions. We think this adds to, rather than detracts from, the value of the paper: the conductor assertion of Lemma \ref{lem:thetachars}, which by Remark \ref{rem:notperiodic} is what makes the sums of Section \ref{sec:salieweights} genuine Sali\'e-type sums at their moduli, is precisely the kind of detail that is lost when the even case is quoted from one source and the odd case from another.
\end{remark}

We can now record the evaluation of $G$ on a divisor class, in the exact form Section \ref{sec:salieweights} requires.

\begin{proposition}[Evaluation on a divisor class]\label{prop:classeval}
Let $r\in\N$ and $a\in\Z$, put $d:=(a,r)$, $m:=r/d$, and write $a=dc$, so that $(c,m)=1$. Write $m=2^\nu m_{\mathrm{odd}}$. Then for every $b\in\Z$:
\begin{itemize}
\item[(i)] $G(a,b,r)=0$ unless $d\mid b$. Write $b=db'$.
\item[(ii)] There are a complex number $\Theta$ with $|\Theta|\in\{0,1,\sqrt2\}$, an odd integer $u$, and an integer $A$, all three determined by $(r,d,b)$ alone and independent of $c$, such that (in particular, whether $G(a,b,r)$ vanishes depends only on $(r,d,b)$)
\begin{equation}
G(a,b,r)=\Theta\,d\,\sqrt m\;\Kr cm\,\vartheta_\nu(uc)\;e_m\!\big(A\bar c\big).
\tag{3.8}
\end{equation}
\item[(iii)] Consequently there are complex numbers $\mu_1,\mu_2$ with $|\mu_1|+|\mu_2|\le4$ and characters $\psi_1,\psi_2\in\{\mathbf1,\chi_{-4},\chi_8\}$, again independent of $c$, such that
\begin{equation}
G(a,b,r)=d\,\sqrt m\;\sum_{i=1,2}\mu_i\,\Kr cm\,\psi_i(c)\;e_m\!\big(A\bar c\big),
\tag{3.9}
\end{equation}
where each $\Kr\cdot m\psi_i$ is a Dirichlet character of conductor dividing $m$ (Lemma \ref{lem:thetachars}); moreover $\psi_1=\psi_2=\mathbf1$ when $m$ is odd, and $\psi_1=\psi_2=\chi_8$ when $m\equiv2\pmod4$. In particular $|G(a,b,r)|\le\sqrt2\,d\sqrt m$.
\item[(iv)] If $d\mid b$ then $(A,m)\le(b',m)^2$.
\end{itemize}
\end{proposition}

The proof is a prime-by-prime application of the evaluations of Sections \ref{sec:gaussodd}--\ref{sec:gauss2}, followed by reassembly through (3.1). We write it out because the precise form of the $c$-dependence in (3.8) -- a Kronecker symbol, a function of $uc\bmod8$, and an additive character in $\bar c$, and nothing else -- is what Section \ref{sec:salieweights} rests on.

\begin{proof}
Write $r=\prod_pp^{f_p}$, $d=\prod_pp^{\delta_p}$ (so $\delta_p\le f_p$) and $m=\prod_pp^{n_p}$ with $n_p:=f_p-\delta_p$. Applying (3.1) once for each prime,
\[
G(a,b,r)=\prod_{p\mid r}G\big(a_p,b,p^{f_p}\big),\qquad a_p:=a\cdot\frac r{p^{f_p}} .
\]
Since $r/p^{f_p}$ is a unit modulo $p$, $v_p(a_p)=v_p(a)\ge\delta_p$, with equality when $n_p>0$ (then $p\mid m$ and $(c,p)=1$); in all cases $a_p=p^{\delta_p}c_p$ with
\begin{equation}
c_p:=c\cdot u_p,\qquad u_p:=\frac d{p^{\delta_p}}\cdot\frac r{p^{f_p}},
\tag{3.10}
\end{equation}
and $u_p$ is a unit modulo $p$; if $n_p>0$ then $c_p$ is a unit modulo $p$.

\emph{The factor at a prime $p\mid r$.} If $n_p=0$, Lemma \ref{lem:degauss} gives $G(a_p,b,p^{f_p})=p^{f_p}\ind[p^{f_p}\mid b]=p^{\delta_p}\ind[p^{\delta_p}\mid b]$, independent of $c$. If $n_p>0$, Lemma \ref{lem:degauss} gives $G(a_p,b,p^{f_p})=0$ unless $p^{\delta_p}\mid b$, and then
\[
G(a_p,b,p^{f_p})=p^{\delta_p}\,G\big(c_p,\,b_p,\,p^{n_p}\big),\qquad b_p:=b/p^{\delta_p}=b'\cdot\frac d{p^{\delta_p}} .
\]
Since $p^{\delta_p}\mid b$ for every $p\mid r$ is equivalent to $d\mid b$, this proves (i), and from now on $d\mid b$.

For odd $p$ with $n_p>0$, Lemma \ref{lem:completesquare} and Proposition \ref{prop:generalgauss} give
\[
G(c_p,b_p,p^{n_p})=e_{p^{n_p}}\!\big({-}\overline{4c_p}\,b_p^2\big)\,G(c_p,0,p^{n_p})=\eta_p\,p^{n_p/2}\,\Kr{c_p}p^{\,n_p}\,e_{p^{n_p}}\!\big(B_p\,\overline{c_p}\big),\qquad B_p:=-\bar4\,b_p^2,
\]
where $\eta_p=\epsilon_p$ if $n_p$ is odd and $\eta_p=1$ if $n_p$ is even, and $\bar4$ is the inverse of $4$ modulo $p^{n_p}$.

For $p=2$ with $n_2=\nu\ge2$, Lemma \ref{lem:twoadic}(i) gives $G(c_2,b_2,2^\nu)=0$ if $b_2$ is odd, and for $b_2=2b_0$ parts (ii) and (iii) together with (3.6) give
\[
G(c_2,b_2,2^\nu)=e_{2^\nu}\!\big({-}\overline{c_2}\,b_0^2\big)\,G(c_2,0,2^\nu)=\sqrt2\,2^{\nu/2}\,\Kr{c_2}{2^\nu}\,\vartheta_\nu(c_2)\,e_{2^\nu}\!\big(B_2\,\overline{c_2}\big),\qquad B_2:=-b_0^2 .
\]
For $p=2$ with $\nu=1$, Lemma \ref{lem:twoadic}(i) gives $G(c_2,b_2,2)=0$ if $b_2$ is even, and $G(c_2,b_2,2)=1+(-1)^{c_2+b_2}=2$ if $b_2$ is odd; since $\vartheta_1=\Kr\cdot2$ and $\Kr{c_2}2^{\,2}=1$, we may write this as $2=-\sqrt2\,2^{1/2}\,\Kr{c_2}2\vartheta_1(c_2)\,e_2(B_2\overline{c_2})$ with $B_2:=1$, using $e_2(\overline{c_2})=-1$ for odd $c_2$. (The choice $B_2=1$ rather than $0$ is made so that (iv) holds; see below.)

\emph{Reassembly.} If any factor vanishes, $G(a,b,r)=0$ and (3.8)--(3.9) hold with $\Theta=0$, $\mu_1=\mu_2=0$, $u=1$, $A:=1$ and $\psi_1=\psi_2:=\chi_8$ if $\nu=1$, $:=\mathbf1$ otherwise (so that the conductor condition in (iii) holds), and (iv) holds trivially. Otherwise multiply the factors. The powers of $p$ give $\prod_pp^{\delta_p}\prod_pp^{n_p/2}=d\sqrt m$. The constants give $\Theta:=\prod_{p\ \mathrm{odd},\ p\mid m}\eta_p\cdot\theta_2$, where $\theta_2=1$ if $\nu=0$, $\theta_2=\sqrt2$ if $\nu\ge2$ and $\theta_2=-\sqrt2$ if $\nu=1$; so $|\Theta|=1$ if $m$ is odd and $|\Theta|=\sqrt2$ if $m$ is even. For the Kronecker symbols, complete multiplicativity in the upper argument gives $\Kr{c_p}p=\Kr cp\Kr{u_p}p$, so
\[
\prod_{p\ \mathrm{odd}}\Kr{c_p}p^{\,n_p}\cdot\Kr{c_2}{2^\nu}=\prod_{p\mid m}\Kr cp^{\,n_p}\cdot\prod_{p\mid m}\Kr{u_p}p^{\,n_p}=\sigma\,\Kr cm,\qquad \sigma:=\prod_{p\mid m}\Kr{u_p}p^{\,n_p}\in\{\pm1\},
\]
by the definition of the Kronecker symbol in Section \ref{sec:kronecker}. The factor $\vartheta_\nu(c_2)=\vartheta_\nu(u_2c)$ occurs only when $\nu\ge1$; we put $u:=u_2$ if $\nu\ge1$ and $u:=1$ if $\nu=0$; in either case $u$ is odd (for $\nu\ge1$ because $r/2^{f_2}$ and $d/2^{\delta_2}$ are odd), which is all that is needed for $\vartheta_\nu(uc)$ and $\psi_i(uc)=\psi_i(u)\psi_i(c)$ below; when $\nu=0$, $\vartheta_0\equiv1$ and $u$ plays no role. Absorbing $\sigma$ into $\Theta$ leaves $|\Theta|$ unchanged. Finally, for each $p\mid m$ the inverse of $c_p$ modulo $p^{n_p}$ is $\bar c\,\overline{u_p}$, where $\bar c$ is the inverse of $c$ modulo $m$ (hence modulo $p^{n_p}$), so the local phases are $e_{p^{n_p}}(B_p\overline{u_p}\,\bar c)$, and
\[
\prod_{p\mid m}e_{p^{n_p}}\big(B_p\overline{u_p}\,\bar c\big)=e_m\big(A\bar c\big),\qquad A:=\sum_{p\mid m}B_p\,\overline{u_p}\,\frac m{p^{n_p}}\in\Z ,
\]
since $\sum_p B_p\overline{u_p}/p^{n_p}=A/m$. Here $B_p$, $u_p$, $\overline{u_p}$ are determined by $(r,d,b)$, so $A$ is independent of $c$. This proves (3.8).

For (iii), the explicit formulas in Lemma \ref{lem:thetachars} and its proof give $\vartheta_\nu(x)=\lambda_1\psi_1(x)+\lambda_2\psi_2(x)$ for all odd $x$, with $|\lambda_1|+|\lambda_2|\le2$ (namely $\vartheta_0=\mathbf1$, $\vartheta_1=\chi_8$, $\vartheta_\nu=\tfrac1{\sqrt2}\mathbf1+\tfrac i{\sqrt2}\chi_{-4}$ for $\nu\ge2$), and $\psi_i(uc)=\psi_i(u)\psi_i(c)$ by complete multiplicativity. Hence (3.9) holds with $\mu_i:=\Theta\lambda_i\psi_i(u)$, and $|\mu_1|+|\mu_2|\le\sqrt2\cdot2<4$. In particular $\psi_1=\psi_2=\mathbf1$ when $\nu=0$ and $\psi_1=\psi_2=\chi_8$ when $\nu=1$ (with the convention $\psi_2:=\psi_1$ of Lemma \ref{lem:thetachars} when $\lambda_2=0$), and the same holds in the vanishing case by the choice made there. The bound $|G(a,b,r)|\le\sqrt2\,d\sqrt m$ is immediate from (3.8).

For (iv), we may assume $G(a,b,r)\ne0$, the vanishing case having been settled. Fix $p\mid m$. Since $m/p^{n_p}$ and $\overline{u_p}$ are units modulo $p^{n_p}$, $A\equiv B_p\overline{u_p}\,(m/p^{n_p})\pmod{p^{n_p}}$ gives $(A,p^{n_p})=(B_p,p^{n_p})$. For odd $p$, $B_p=-\bar4b_p^2$ with $\bar4$ a unit, and $b_p=b'\,(d/p^{\delta_p})$ with $d/p^{\delta_p}$ a unit modulo $p$, so $(B_p,p^{n_p})=(b'^2,p^{n_p})$. For $p=2$ with $\nu\ge2$, $B_2=-b_0^2$ with $2b_0=b_2=b'(d/2^{\delta_2})$ and $d/2^{\delta_2}$ odd, so $(B_2,2^\nu)=(b_0^2,2^\nu)\le(b_2^2,2^\nu)=(b'^2,2^\nu)$. For $p=2$ with $\nu=1$, $B_2=1$, so $(A,2)=1\le(b'^2,2)$. Multiplying over $p\mid m$ gives $(A,m)\le(b'^2,m)$, and $(b'^2,m)\le(b',m)^2$ because at each prime $\min(2v_p(b'),v_p(m))\le2\min(v_p(b'),v_p(m))$.
\end{proof}

We verified Proposition \ref{prop:classeval} numerically in the explicit form (3.8), with $\Theta$, $u$ and $A$ computed by the recipe of the proof. For each of the $31$ moduli
\[
\begin{gathered}r\in\{6,10,12,18,24,27,30,40,45,48,50,54,72,75,90,96,\\100,105,121,135,144,150,189,200,225,245,250,288,360,441,512\}\end{gathered}
\]
-- odd, even, squarefree, prime-power and highly $2$-divisible, including many with $m\equiv2\ (4)$ -- for every divisor $d\mid r$ with $m=r/d\ge2$, for a spread of $b\bmod r$, and for every unit $c\bmod m$, we compared $G(dc,b,r)$ computed from its definition with the right-hand side of (3.8). Over $32{,}167$ such configurations the worst relative discrepancy was $3.6\times10^{-9}$, the predicted vanishing occurred exactly when predicted, and $(A,m)\le(b',m)^2$ held in every non-vanishing case.

\begin{remark}[Odd $m$]\label{rem:Aodd}
If $m$ is odd and $G(a,b,r)\ne0$, the proof gives more than (iv): $A\equiv b'^2v\pmod m$ for some $v$ coprime to $m$. Indeed, for each $p\mid m$ we have $A\equiv B_p\overline{u_p}\,(m/p^{n_p})\pmod{p^{n_p}}$ with $B_p=-\bar4\,b_p^2$ and $b_p=b'\,(d/p^{\delta_p})$, so that
\[
A\equiv b'^2\cdot\Big({-}\bar4\,\big(d/p^{\delta_p}\big)^2\,\overline{u_p}\,\frac m{p^{n_p}}\Big)\pmod{p^{n_p}},
\]
and the factor in brackets is a unit modulo $p$; by the Chinese remainder theorem these local units combine to a unit $v$ modulo $m$.
\end{remark}
\section{Short Sali\'e sums with smooth weights}\label{sec:salieweights}

This section contains the only use of Hypothesis \ref{hyp:R*}. Lemma \ref{lem:shortsalie} brings the characters produced by the Gauss sums into the form of the hypothesis, Lemma \ref{lem:smoothsalie} adds a smooth weight, and Proposition \ref{prop:hsum} combines this with the evaluation of the Gauss sums.

\begin{lemma}[Short sums of Sali\'e type]\label{lem:shortsalie}
Let $m\in\N$, $A,B\in\Z$, $\psi\in\{\mathbf1,\chi_{-4},\chi_8\}$ and $\chi(t):=\Kr tm\psi(t)$. Then, under Hypothesis \ref{hyp:R*}, for any real $x_1\le x_2$ with $x_2-x_1\le m$,
\[
\sum_{x_1<t\le x_2}\chi(t)\,e_m\big(A\bar t+Bt\big)\ \ll\ (x_2-x_1+1)^{1/2}\,m^\eps\,(A,m)^{1/2},
\]
the terms with $(t,m)>1$ being zero.
\end{lemma}

\begin{proof}
If $\psi=\mathbf1$ this is Hypothesis \ref{hyp:R*}. If $\psi\ne\mathbf1$, then $\psi$ vanishes at even arguments, so $\chi(t)=0$ for even $t$, and for odd $t$ coprime to $m$ one has $\Kr t{8m}=\Kr t2^{\,3}\Kr tm=\chi_8(t)\Kr tm$. Hence
\[
\chi(t)=\psi'(t)\,\Kr t{8m},\qquad \psi':=\psi\chi_8 ,
\]
for all $t\in\Z$, both sides vanishing at even $t$. Expand the character $\psi'$ modulo $8$ in additive characters, $\psi'(t)=\sum_{k\bmod 8}\hat c_k e_8(kt)$ with $|\hat c_k|\le1$. Let $\tilde t$ be the inverse of $t$ modulo $8m$; then $\tilde t\equiv\bar t\pmod m$, so $e_m(A\bar t)=e_{8m}(8A\tilde t)$, and $e_m(Bt)e_8(kt)=e_{8m}((8B+km)t)$. Therefore
\[
\sum_{x_1<t\le x_2}\chi(t)e_m(A\bar t+Bt)=\sum_{k\bmod8}\hat c_k\sum_{x_1<t\le x_2}\Kr t{8m}\,e_{8m}\big(8A\,\tilde t+(8B+km)t\big).
\]
Each inner sum runs over an interval of length $\le m\le8m$, so Hypothesis \ref{hyp:R*} at the modulus $8m$ bounds it by $\ll(x_2-x_1+1)^{1/2}m^\eps(8A,8m)^{1/2}\ll(x_2-x_1+1)^{1/2}m^\eps(A,m)^{1/2}$.
\end{proof}

Lemma \ref{lem:shortsalie} uses Hypothesis \ref{hyp:R*} exactly as stated, for the Kronecker symbol and no other character. This is possible because of Lemma \ref{lem:thetachars}: the characters that arise are $\Kr\cdot m$ times one of three explicit characters modulo $8$, and passing from $m$ to $8m$ absorbs them back into a Kronecker symbol. A modulus $m\equiv2\pmod4$ always comes with $\psi=\chi_8$ (Proposition \ref{prop:classeval}(iii)), so the hypothesis is never invoked at a modulus where the Kronecker symbol fails to be a character. Under Hypothesis \ref{hyp:Cprime} the lemma is immediate, since $\chi$ is a Dirichlet character modulo $[m,8]$.

\begin{lemma}[Smooth weights]\label{lem:smoothsalie}
Assume Hypothesis \ref{hyp:R*}. Let $m\in\N$, $A\in\Z$, $\psi\in\{\mathbf1,\chi_{-4},\chi_8\}$, $\chi=\Kr\cdot m\psi$, and $0<T\le m$. Then
\[
\sum_{t\ne0}\widehat V\Big(\frac tT\Big)\,\chi(t)\,e_m\big(A\bar t\big)\ \ll\ T^{1/2}\,m^\eps\,(A,m)^{1/2}.
\]
\end{lemma}

\begin{proof}
\emph{Negative $t$.} Since $\widehat V$ is even, $\chi(-t)=\chi(-1)\chi(t)$ and the inverse of $-t$ is $-\bar t$, the terms with $t<0$ form $\chi(-1)\sum_{t\ge1}\widehat V(t/T)\chi(t)e_m(-A\bar t)$, a sum of the same kind with $-A$ in place of $A$. So it suffices to treat $t\ge1$.

\emph{The partial sums.} Put $a(t):=\chi(t)e_m(A\bar t)$ and $S(u):=\sum_{1\le t\le u}a(t)$. By Lemma \ref{lem:shortsalie} with $B=0$,
\[
|S(u)|\ll u^{1/2}m^\eps(A,m)^{1/2}\qquad(1\le u\le m).
\]
For $u>m$, split $[1,u]$ into $\lfloor u/m\rfloor$ intervals of length $m$ and one shorter interval, and apply Lemma \ref{lem:shortsalie} to each:
\[
|S(u)|\ll\Big(\frac um+1\Big)m^{1/2+\eps}(A,m)^{1/2}\ll u\,m^{-1/2+\eps}(A,m)^{1/2}\qquad(u>m).
\]

\emph{Partial summation.} Let $g(u):=\widehat V(u/T)$. Since $S(u)\ll u$ and $g$ decays rapidly,
\[
\sum_{t\ge1}g(t)a(t)=-\int_1^\infty S(u)\,g'(u)\,du,\qquad g'(u)=T^{-1}\,\widehat V'(u/T).
\]
On $[1,m]$, the substitution $u=Tv$ gives
\[
\int_1^m|S(u)g'(u)|\,du\ \ll\ m^\eps(A,m)^{1/2}\,T^{1/2}\int_0^\infty v^{1/2}\,|\widehat V'(v)|\,dv\ \ll\ T^{1/2}m^\eps(A,m)^{1/2}.
\]
On $[m,\infty)$ we have $u/T\ge m/T\ge1$, so $|\widehat V'(u/T)|\ll(T/u)^3$ by \eqref{eq:decay}, and
\[
\int_m^\infty|S(u)g'(u)|\,du\ \ll\ m^{-1/2+\eps}(A,m)^{1/2}\int_m^\infty u\cdot\frac{T^2}{u^3}\,du\ =\ \frac{T^2}{m^{3/2}}\,m^\eps(A,m)^{1/2}\ \le\ T^{1/2}m^\eps(A,m)^{1/2},
\]
using $T\le m$.
\end{proof}

We now combine Lemma \ref{lem:smoothsalie} with the evaluation of the Gauss sums. The following proposition is the arithmetic heart of the paper.

\begin{proposition}\label{prop:hsum}
Assume Hypothesis \ref{hyp:R*}. Let $r\in\N$, $b\in\Z$ with $(b,r)=1$, $l\in\Z$, and $W\ge1$. Then
\[
\bigg|\sum_{h\ne0}\widehat V\Big(\frac{hW}r\Big)\,G(-hb,l,r)\bigg|\ \ll\ \frac{r^{1+\eps}}{W^{1/2}}\sum_{d\mid(r,l)}\Big(\frac ld,\frac rd\Big).
\]
\end{proposition}

\begin{proof}
\emph{The classes.} Group the $h\ne0$ according to $d:=(h,r)$, a divisor of $r$, and put $m:=r/d$. The $h$ with $(h,r)=d$ are exactly the $h=dh'$ with $h'\ne0$ and $(h',m)=1$.

\emph{The Gauss sum on a class.} Fix $d$ and such an $h=dh'$. We apply Proposition \ref{prop:classeval} with its $a$ equal to $-hb=d\cdot(-h'b)$ and its $b$ equal to $l$. Since $(b,r)=1$ and $(h',m)=1$, we have $(-hb,r)=d$, so the $d$ and $m$ of the proposition are ours, and its $c$ is $c=-h'b$. By part (i), $G(-hb,l,r)=0$ unless $d\mid l$. If $d\mid l$, part (iii) gives
\[
G(-hb,l,r)=d\sqrt m\sum_{i=1,2}\mu_i\Kr cm\psi_i(c)\,e_m\big(A\bar c\big),
\]
where $|\mu_1|+|\mu_2|\le4$, $\psi_i\in\{\mathbf1,\chi_{-4},\chi_8\}$, and $\mu_i,\psi_i,A$ depend only on $(r,d,l)$, not on $h'$; by part (iv), $(A,m)\le(l/d,m)^2$.

\emph{Separating $h'$.} By complete multiplicativity of the Kronecker symbol and of $\psi_i$,
\[
\Kr cm\psi_i(c)=\kappa_i\,\Kr{h'}m\psi_i(h'),\qquad \kappa_i:=\Kr{-b}m\psi_i(-b),
\]
and $|\kappa_i|=1$: $(b,m)=1$, and if $\psi_i\ne\mathbf1$ then $m$, hence $r$, is even and $b$ is odd. Moreover $\bar c\equiv-\bar b\,\overline{h'}\pmod m$, so $e_m(A\bar c)=e_m(A_1\overline{h'})$ with $A_1:=-A\bar b$ and $(A_1,m)=(A,m)$. The Kronecker symbol vanishes when $(h',m)>1$, so the restriction $(h',m)=1$ may be dropped, and the contribution of the class $d$ is
\[
d\sqrt m\sum_{i=1,2}\mu_i\kappa_i\sum_{h'\ne0}\widehat V\Big(\frac{h'}T\Big)\Kr{h'}m\psi_i(h')\,e_m\big(A_1\overline{h'}\big),\qquad T:=\frac mW ,
\]
since $hW/r=h'dW/r=h'W/m=h'/T$.

\emph{Applying the hypothesis.} As $W\ge1$ we have $T\le m$, and Lemma \ref{lem:smoothsalie} bounds the inner sum by
\[
\ll\Big(\frac mW\Big)^{1/2}m^\eps(A,m)^{1/2}\le\Big(\frac mW\Big)^{1/2}m^\eps\Big(\frac ld,m\Big).
\]
The class $d$ therefore contributes $\ll d\sqrt m\,(m/W)^{1/2}m^\eps(l/d,m)=r\,W^{-1/2}m^\eps(l/d,m)$. Summing over the divisors $d$ of $r$ that divide $l$ gives the claim.
\end{proof}

\begin{remark}[Odd moduli: Conjecture (C$^*$) suffices]\label{rem:Cstar}
If $r$ is odd, Proposition \ref{prop:hsum} holds under Conjecture (C$^*$) of Friedlander and Iwaniec, \eqref{eq:Cstar}, in place of Hypothesis \ref{hyp:R*}. Indeed, classes on which the Gauss sum vanishes contribute nothing; on the other classes, every $m=r/d$ is odd, so $\psi_i=\mathbf1$, the Kronecker symbol is the Jacobi symbol, and by Remark \ref{rem:Aodd} $A_1\equiv(l/d)^2v'\pmod m$ with $(v',m)=1$. The partial sums $S(u)$ in the proof of Lemma \ref{lem:smoothsalie} are then exactly the sums of \eqref{eq:Cstar}, with $q=m$, $k=l/d$, $M=u$ and $a$ a positive representative of $\pm v'$, and \eqref{eq:Cstar} gives
\[
S(u)\ll\big(u^{1/2}+u\,m^{-1/2}\big)\big((l/d)^2,m\big)^{1/2}m^\eps\le\big(u^{1/2}+u\,m^{-1/2}\big)\Big(\frac ld,m\Big)m^\eps
\]
for all $u\ge1$, which is all the proof uses. In particular, no linear phase and no arbitrary intervals are needed at odd moduli. At even moduli the linear phase $e_{8m}(kmt)$ enters through the absorption of the characters modulo $8$ in Lemma \ref{lem:shortsalie}. Since the denominators $r$ in the proof of Theorem \ref{thm:P} come from Dirichlet's theorem and may be even, the theorem itself needs Hypothesis \ref{hyp:R*} for even moduli as well.
\end{remark}

\section{Counting fractions with square denominators}\label{sec:counting}

Throughout this section $Q,N\ge1$ and $\Delta=1/N$, and $P(\alpha)$ is as in \eqref{eq:Pdef}.

\subsection{From the large sieve to the counting function}

\begin{lemma}\label{lem:LS}
For all $Q\ge1$, $N\ge2$, $M\in\R$ and complex $a_n$,
\[
S(Q,M,N,(a_n))\ \le\ \frac{\pi^2}2\,N\,Z\cdot\max_{\alpha\in\R}P(\alpha).
\]
\end{lemma}

\begin{proof}
Note that $\Delta=1/N\le\frac12$.

\emph{The points.} Let $\mathcal I$ be the set of pairs $s=(q,a)$ with $1\le q\le Q$, $1\le a\le q^2$, $(a,q)=1$, and put $\alpha_s:=a/q^2$. Then $S=\sum_{s\in\mathcal I}|\sum_na_ne(n\alpha_s)|^2$. For $s,t\in\mathcal I$ write $s\sim t$ if $\|\alpha_s-\alpha_t\|\le\Delta$, where $\|\cdot\|$ is the distance to the nearest integer. For fixed $t$, the number of $s\sim t$ is at most $P(\alpha_t)$: if $s=(q,a)\sim t$, then some representative $a'$ of the class of $a$ modulo $q^2$ satisfies $|a'/q^2-\alpha_t|\le\Delta$, and $(a',q)=(a,q)=1$, so $(q,a')$ is counted by $P(\alpha_t)$; distinct $s$ give distinct pairs $(q,a')$.

\emph{Duality.} The operator norm of a matrix equals that of its adjoint. Hence it suffices to show
\[
\sum_{M<n\le M+N}\Big|\sum_{s\in\mathcal I}b_se(n\alpha_s)\Big|^2\ \le\ \frac{\pi^2}2\,N\max_{\alpha}P(\alpha)\sum_s|b_s|^2
\]
for all complex $b_s$.

\emph{A majorant.} Let $F(x):=(\sin\pi x/(\pi x))^2$. Then $F\ge0$, $\widehat F(\xi)=\max\{0,1-|\xi|\}$, and $F(x)\ge4/\pi^2$ for $|x|\le\frac12$, because $\sin\pi x\ge2x$ on $[0,\frac12]$. Put $x_0:=M+N/2$ and $\phi(x):=\frac{\pi^2}4F((x-x_0)/N)$. Then $\phi\ge0$, $\phi\ge1$ on $[M,M+N]$, and
\[
\widehat\phi(\xi)=\frac{\pi^2}4\,N\,\widehat F(N\xi)\,e(-x_0\xi),\qquad |\widehat\phi(\xi)|\le\frac{\pi^2}4N,\qquad \widehat\phi(\xi)=0\ \text{ for }|\xi|>\Delta .
\]

\emph{Poisson summation.} For $\theta\in\R$, the function $x\mapsto\phi(x)e(x\theta)$ is continuous and $O((1+|x|)^{-2})$, and its Fourier transform $\widehat\phi(\xi-\theta)$ is continuous with compact support, so Poisson summation gives
\[
\sum_{n\in\Z}\phi(n)e(n\theta)=\sum_{k\in\Z}\widehat\phi(k-\theta).
\]
At most two $k$ satisfy $|k-\theta|\le\Delta\le\frac12$, and none unless $\|\theta\|\le\Delta$. Hence $|\sum_n\phi(n)e(n\theta)|\le\frac{\pi^2}2N\,\ind[\|\theta\|\le\Delta]$, and
\[
\sum_{M<n\le M+N}\Big|\sum_sb_se(n\alpha_s)\Big|^2\le\sum_{s,t}b_s\overline{b_t}\sum_n\phi(n)e\big(n(\alpha_s-\alpha_t)\big)\le\frac{\pi^2}2N\sum_{s\sim t}|b_s||b_t| .
\]
Finally $|b_sb_t|\le\frac12(|b_s|^2+|b_t|^2)$ and the relation $\sim$ is symmetric, so $\sum_{s\sim t}|b_s||b_t|\le\sum_s|b_s|^2\cdot\max_t\#\{s:s\sim t\}\le\sum_s|b_s|^2\max_{\alpha}P(\alpha)$.
\end{proof}

\subsection{The elementary count}

The following lemma is the classical way of counting $P(\alpha)$; we state it for every denominator and every $z$.

\begin{lemma}[The elementary count]\label{lem:E}
Let $r\in\N$, $b\in\Z$ with $(b,r)=1$, $z\in\R$, $\alpha=b/r+z$, and
\[
X:=Q^2r\big(|z|+\Delta\big).
\]
Then
\[
P(\alpha)\ \le\ 1+\sum_{1\le|m|\le X}\rho_r(-\bar bm)\Big(\frac{Q^3\Delta}{|m|}+1\Big)\ \ll\ \big(1+X+Q^3\Delta\big)\,\big(r(1+X)\big)^\eps .
\]
\end{lemma}

\begin{proof}
\emph{The map.} For a pair $(q,a)$ counted by $P(\alpha)$ put $m:=ar-bq^2$. Then $m\equiv-bq^2\pmod r$, that is, $q^2\equiv-\bar bm\pmod r$. Since
\[
\frac a{q^2}-\alpha=\frac{m-q^2rz}{q^2r},
\]
the condition $|a/q^2-\alpha|\le\Delta$ reads $|m-q^2rz|\le q^2r\Delta$, which implies $|m|\le q^2r(|z|+\Delta)\le X$. The map $(q,a)\mapsto(q,m)$ is injective, since $a=(m+bq^2)/r$.

\emph{The value $m=0$.} If $m=0$ then $ar=bq^2$. As $(a,q)=1$ we get $q^2\mid r$, and as $(b,r)=1$ we get $r\mid q^2$. So $q^2=r$ and $a=b$, and at most one pair has $m=0$.

\emph{A fixed $m\ne0$.} The $q$ that occur with this $m$ lie in $\rho_r(-\bar bm)$ residue classes modulo $r$. They also lie in the set
\[
J_m:=\big\{t\in[1,Q]:\ |g(t)-z|\le\Delta\big\},\qquad g(t):=\frac m{t^2r},
\]
because $|m-q^2rz|\le q^2r\Delta$ is the same as $|g(q)-z|\le\Delta$. The function $g$ is monotone on $[1,Q]$ with $|g'(t)|=2|m|/(t^3r)\ge2|m|/(Q^3r)$, so $J_m$ is an interval of length at most
\[
\frac{2\Delta}{2|m|/(Q^3r)}=\frac{Q^3\Delta\,r}{|m|}.
\]
A residue class modulo $r$ therefore contains at most $Q^3\Delta/|m|+1$ elements of $J_m$. This proves the first inequality.

\emph{Summation.} If $X<1$ the sum is empty. Otherwise Lemma \ref{lem:rho}, applied with $j=-\bar b$, gives $\sum_{1\le|m|\le X}\rho_r(-\bar bm)\ll Xr^\eps$ and $\sum_{1\le|m|\le X}\rho_r(-\bar bm)/|m|\ll r^\eps\log(2X)$.
\end{proof}

When $|z|\ge\Delta$, $N=Q^3$ and $r\le Q^{O(1)}$, this is the familiar bound $P(\alpha)\ll(1+Q^2r|z|)Q^\eps$. The point of Lemma \ref{lem:E} is that it holds for every $r$ and every $z$.

\subsection{Zhao's smoothing at a good approximation}

We now treat the case $|z|\le\Delta$, in which the elementary count gives $1+Q^2r\Delta$, and show that Hypothesis \ref{hyp:R*} gives its square root. For $Q_1\ge\frac12$ let $P_{Q_1}(\alpha)$ be the number of pairs counted by $P(\alpha)$ with $Q_1<q\le2Q_1$.

\begin{proposition}\label{prop:F}
Assume Hypothesis \ref{hyp:R*}. Let $r\in\N$, $b\in\Z$ with $(b,r)=1$, $|z|\le\Delta$, and $\alpha=b/r+z$. Then for every $Q_1\ge\frac12$,
\[
P_{Q_1}(\alpha)\ \ll\ \Big(1+Q_1^3\Delta+\big(Q_1^2r\Delta\big)^{1/2}\Big)\,r^\eps,
\]
and consequently
\[
P(\alpha)\ \ll\ \Big(1+Q^3\Delta+\big(Q^2r\Delta\big)^{1/2}\Big)\,r^\eps\log(2Q).
\]
\end{proposition}

\begin{proof}
\emph{The window.} Put $W:=16Q_1^2r\Delta$. For a pair counted by $P_{Q_1}(\alpha)$, the integer $m=ar-bq^2$ of Lemma \ref{lem:E} satisfies
\[
|m|\le q^2r(|z|+\Delta)\le4Q_1^2\cdot r\cdot2\Delta=\frac W2 .
\]
If $W<1$, then $m=0$, and $P_{Q_1}(\alpha)\le1$ by Lemma \ref{lem:E}. Let therefore $W\ge1$.

\emph{The smooth majorant.} The weights $w(q/Q_1)$ and $V(m/W)$ of Section \ref{sec:fourier} both equal $1$ at a counted pair. Since the map $(q,a)\mapsto(q,m)$ is injective and at most one pair has $m=0$,
\begin{equation}\label{eq:Omegastar}
P_{Q_1}(\alpha)\ \le\ 1+\Omega^*,\qquad
\Omega^*:=\sum_{q\in\Z}w\Big(\frac q{Q_1}\Big)\sum_{\substack{m\equiv-bq^2\ (\mathrm{mod}\ r)\\ m\ne0}}V\Big(\frac mW\Big).
\end{equation}
Let $\Omega$ denote the same double sum with the term $m=0$ included.

\emph{Double Poisson summation.} Group $q$ according to its residue class $k\bmod r$:
\[
\Omega=\sum_{k\bmod r}\Big(\sum_{q\equiv k\ (r)}w\Big(\frac q{Q_1}\Big)\Big)\Big(\sum_{m\equiv-bk^2\ (r)}V\Big(\frac mW\Big)\Big).
\]
By \eqref{eq:poissonAP},
\[
\sum_{q\equiv k\ (r)}w\Big(\frac q{Q_1}\Big)=\frac{Q_1}r\sum_{l\in\Z}\widehat w\Big(\frac{lQ_1}r\Big)e_r(lk),\qquad
\sum_{m\equiv-bk^2\ (r)}V\Big(\frac mW\Big)=\frac Wr\sum_{h\in\Z}\widehat V\Big(\frac{hW}r\Big)e_r\big({-}hbk^2\big).
\]
In the term $l=0$ we keep the sum over $m$ as it is; in the terms $l\ne0$ we insert the second expansion and sum over $k$. Since $\sum_{k\bmod r}e_r(-hbk^2+lk)=G(-hb,l,r)$, and since the number of $k\bmod r$ with $-bk^2\equiv m$ is $\rho_r(-\bar bm)$, we obtain the exact identity
\begin{align*}
\Omega&=\Omega_0+\Omega_1,\\
\Omega_0&:=\frac{Q_1\widehat w(0)}r\sum_{m\in\Z}\rho_r(-\bar bm)\,V\Big(\frac mW\Big),\\
\Omega_1&:=\frac{Q_1W}{r^2}\sum_{l\ne0}\widehat w\Big(\frac{lQ_1}r\Big)\sum_{h\in\Z}\widehat V\Big(\frac{hW}r\Big)G(-hb,l,r).
\end{align*}
All sums converge absolutely, by \eqref{eq:decay}.

\emph{Removing $m=0$.} The terms of $\Omega$ with $m=0$ are those with $r\mid bq^2$, i.e.\ $r\mid q^2$, i.e.\ $q\in r_0\Z$ with $r_0:=\prod_{p^e\|r}p^{\lceil e/2\rceil}$. Note that $\rho_r(0)=\#\{k\bmod r:r_0\mid k\}=r/r_0$. By \eqref{eq:poissonAP} with modulus $1$,
\[
\Omega-\Omega^*=V(0)\sum_{n\in\Z}w\Big(\frac{r_0n}{Q_1}\Big)=\frac{Q_1}{r_0}\sum_{n\in\Z}\widehat w\Big(\frac{nQ_1}{r_0}\Big).
\]
The term $n=0$ equals $Q_1\widehat w(0)\rho_r(0)/r$, which is the term $m=0$ of $\Omega_0$. Writing $\Omega_0^\sharp$ for $\Omega_0$ without its term $m=0$, we therefore have the exact identity
\[
\Omega^*=\Omega_0^\sharp+\Omega_1-\frac{Q_1}{r_0}\sum_{n\ne0}\widehat w\Big(\frac{nQ_1}{r_0}\Big),
\]
and by Lemma \ref{lem:fourier}(i) the last sum is $O(1)$. Hence
\begin{equation}\label{eq:Omegasplit}
\Omega^*\ \le\ \Omega_0^\sharp+|\Omega_1|+O(1).
\end{equation}

\emph{The term $\Omega_0^\sharp$.} Since $V(m/W)=0$ for $|m|\ge W$, Lemma \ref{lem:rho} gives
\[
\Omega_0^\sharp\le\frac{2Q_1}r\sum_{1\le|m|\le W}\rho_r(-\bar bm)\ll\frac{Q_1W}r\,r^\eps=16\,Q_1^3\Delta\,r^\eps .
\]

\emph{The terms of $\Omega_1$ with $h=0$.} Here $G(0,l,r)=\sum_{k\bmod r}e_r(lk)=r\,\ind[r\mid l]$. Writing $l=rn$, these terms are bounded by
\[
\frac{Q_1W}{r^2}\cdot\widehat V(0)\cdot r\sum_{n\ne0}\big|\widehat w(nQ_1)\big|\ \le\ \frac{2W}r\cdot Q_1\sum_{n\ne0}|\widehat w(nQ_1)|\ \ll\ \frac Wr=16Q_1^2\Delta\le32\,Q_1^3\Delta ,
\]
by Lemma \ref{lem:fourier}(i) and $Q_1\ge\frac12$.

\emph{The terms of $\Omega_1$ with $h\ne0$.} For each $l\ne0$, Proposition \ref{prop:hsum} applies, since $W\ge1$. Writing $l=dl'$ for each divisor $d$ of $r$ that divides $l$, we get
\[
\Big|\Omega_1^{h\ne0}\Big|\ \ll\ \frac{Q_1W}{r^2}\cdot\frac{r^{1+\eps}}{W^{1/2}}\sum_{d\mid r}\ \sum_{l'\ne0}\Big|\widehat w\Big(\frac{l'dQ_1}r\Big)\Big|\Big(l',\frac rd\Big).
\]
By Lemma \ref{lem:fourier}(ii) with $Y=r/(dQ_1)$, the inner sum is $\ll r^{1+\eps}/(dQ_1)$, and $\sum_{d\mid r}1/d\le\sum_{d\mid r}1\ll r^\eps$. Hence
\[
\Big|\Omega_1^{h\ne0}\Big|\ \ll\ \frac{Q_1W}{r^2}\cdot\frac{r^{1+\eps}}{W^{1/2}}\cdot\frac{r^{1+\eps}}{Q_1}\ =\ W^{1/2}r^{2\eps}.
\]

\emph{Conclusion.} Inserting these bounds into \eqref{eq:Omegasplit} and \eqref{eq:Omegastar} gives $P_{Q_1}(\alpha)\ll(1+Q_1^3\Delta+W^{1/2})r^\eps$, which is the first claim since $W^{1/2}=4(Q_1^2r\Delta)^{1/2}$. For the second, cover $1\le q\le Q$ by the ranges $Q_1<q\le2Q_1$ with $Q_1=Q/2^k$, $1\le k\le\log_2(2Q)$, and sum: $\sum_kQ_1^3\ll Q^3$ and $\sum_kQ_1\ll Q$.
\end{proof}

\begin{remark}[Zhao's normalization, and the analytic term]\label{rem:zhao}
This remark is heuristic and is not used in the proofs. Proposition \ref{prop:F} is Zhao's treatment \cite{Zhao2004} in a different normalization. Since $m=ar-bq^2$ is an affine function of $a$, a smooth sum over $m$ in a residue class modulo $r$ is a smooth sum over $a$. In Zhao's normalization one majorizes the number of $a$ for a given $q$ by $\sum_aV((a-\alpha q^2)/H)$ with $H\asymp Q_1^2\Delta$, centring the window at $\alpha q^2$, that is, at $m=q^2rz$. Poisson summation in $a$ then produces
\[
H\sum_h\widehat V(Hh)\,e_r\big(hbq^2\big)\,e\big(hzq^2\big),
\]
and completing the sum over $q$ modulo $r$ gives the Gauss sums $G(hb,l,r)$ times the exponential integrals
\[
\int_\R w\Big(\frac t{Q_1}\Big)\,e\Big(hzt^2-\frac{lt}r\Big)\,dt ,
\]
with the quadratic phase $hzt^2$. When $|z|\le\Delta$, this phase is $O(Q^\eps)$ on the range $|h|\le Q^\eps/H$, $t\le\frac52Q_1$, outside of which $\widehat V(Hh)$ is negligible; so the integral is non-oscillatory up to a factor $Q^\eps$, and acts essentially as a smooth weight. In Proposition \ref{prop:F} we have instead centred the window at $0$; since the true centre $q^2rz$ moves by at most $W/4$, this costs a factor $2$ in the length of the window and removes the factor $e(hzq^2)$ altogether.

When $|z|$ is much larger than $\Delta$, the quadratic phase oscillates, and stationary phase replaces the integral by a factor $e(-l^2/(4hzr^2))$ with a non-trivial amplitude. This is the situation of the earlier treatments, with denominators $r\le\sqrt N$, and we avoid it by choosing the approximation so that either $|z|\le\Delta$ or the elementary count suffices.
\end{remark}

\section{Proof of the main results}\label{sec:proofs}

\begin{proof}[Proof of Theorem \ref{thm:P}]
Let $\alpha\in\R$ and put
\[
\tau:=Q^{2/3}N^{1/3}\ \ge1 .
\]
By Dirichlet's approximation theorem there are integers $b$ and $1\le r\le\tau$ with $|r\alpha-b|\le1/\tau$; dividing by $(b,r)$ we may assume $(b,r)=1$. Thus
\[
\alpha=\frac br+z,\qquad (b,r)=1,\qquad 1\le r\le\tau,\qquad |z|\le\frac1{r\tau}.
\]
Note that $r^\eps\le\tau^\eps\le(QN)^\eps$.

\emph{Case 1: $|z|>\Delta$.} Then $X=Q^2r(|z|+\Delta)<2Q^2r|z|\le2Q^2/\tau$, and Lemma \ref{lem:E} gives
\[
P(\alpha)\ll\Big(1+\frac{Q^2}\tau+\frac{Q^3}N\Big)(QN)^\eps=\Big(1+\frac{Q^{4/3}}{N^{1/3}}+\frac{Q^3}N\Big)(QN)^\eps .
\]

\emph{Case 2: $|z|\le\Delta$.} Then Proposition \ref{prop:F} gives, since $r\le\tau$,
\[
P(\alpha)\ll\Big(1+\frac{Q^3}N+\Big(\frac{Q^2\tau}N\Big)^{1/2}\Big)(QN)^\eps=\Big(1+\frac{Q^3}N+\frac{Q^{4/3}}{N^{1/3}}\Big)(QN)^\eps .
\]
In both cases \eqref{eq:Pbound} holds.
\end{proof}

The value of $\tau$ is the one that balances the two cases: Case 1 contributes $Q^2/\tau$ and Case 2 contributes $(Q^2\tau/N)^{1/2}$.

\begin{proof}[Proof of Theorem \ref{thm:main}]
If $N<2$, the interval $(M,M+N]$ contains at most two integers, so by Cauchy--Schwarz each inner sum satisfies $|\sum_na_ne(an/q^2)|^2\le2Z$, and $S\le2Z\sum_{q\le Q}q^2\le2Q^3Z$. If $N\ge2$, Lemma \ref{lem:LS} and Theorem \ref{thm:P} give
\[
S(Q,M,N,(a_n))\ \ll\ NZ\cdot(QN)^\eps\Big(1+\frac{Q^3}N+\frac{Q^{4/3}}{N^{1/3}}\Big)=(QN)^\eps\Big(N+Q^3+Q^{4/3}N^{2/3}\Big)Z .
\qedhere
\]
\end{proof}

\begin{proof}[Proof of Corollary \ref{cor:critical}]
For $N=Q^3$ we have $Q^{4/3}N^{2/3}=Q^{10/3}=Q^{1/3}N$, which dominates $Q^3+N=2N$.
\end{proof}

\section{Remarks}\label{sec:remarks}

\subsection{Why denominators beyond \texorpdfstring{$\sqrt N$}{sqrt N}}\label{sec:whytau}

Since \cite{Baier2006,BaierZhao2008}, the counting function has been studied near fractions $b/r$ with $r\le\sqrt N$, obtained by Dirichlet's theorem with $\tau=\sqrt N$. With that choice, Case 1 of the proof of Theorem \ref{thm:P} gives only $Q^2/\sqrt N=Q^{1/2}$ at $N=Q^3$, the classical bound. To do better one has to analyse the $\alpha$ near $b/r$ with $|z|$ between $\Delta$ and $1/(r\sqrt N)$. There the analytic term of Remark \ref{rem:zhao} oscillates, and the arithmetic and the analysis have to be handled together; this is what the earlier versions of this paper did, at considerable length.

Theorem \ref{thm:P} avoids this. Write $z=Q^{-3/2-\gamma}r^{-1}$ at $N=Q^3$, with $r\le Q^{3/2}$ and $\gamma\ge0$. If $b/r$ is the last convergent of the continued fraction of $\alpha$ with denominator at most $Q^{3/2}$, the next convergent $b_+/r_+$ has $r_+\asymp1/(r|z|)=Q^{3/2+\gamma}$ and $|\alpha-b_+/r_+|<1/r_+^2\le\Delta$. So every $\alpha$ is within $\Delta$ of a fraction whose denominator is at most $Q^{3/2+\gamma}$, up to a constant factor. For such a fraction the analytic term does not oscillate, and Proposition \ref{prop:F} gives $Q^{1/4+\gamma/2}$, while Lemma \ref{lem:E} at $b/r$ gives $Q^{1/2-\gamma}$. The two bounds meet at $\gamma=\frac16$, where both are $Q^{1/3}$. The choice $\tau=Q^{5/3}$ in the proof of Theorem \ref{thm:P} is a compact way of organizing this case distinction.

\subsection{What limits the exponent}\label{sec:limit}

At $N=Q^3$ the two cases of the proof meet at $r\asymp Q^{5/3}$. In Case 2, Proposition \ref{prop:F} already gives the full square root of the elementary count. A further gain there would require cancellation beyond Hypothesis \ref{hyp:R*}, which only exploits the sum over $h$; the sum over $l$ is bounded in absolute value. In Case 1, a gain would require a bound below $Q^{1/3}$ for $P(b/r+z)$ when $|z|\asymp Q^{-5/3}r^{-1}$, for every $r\le Q^{5/3}$. In the notation of Section \ref{sec:whytau} this is $\gamma=\frac16$. The earlier versions of this paper contain bounds of this kind for some ranges of $r$, conditional and unconditional; we have not examined whether they can be combined with the present argument to lower the exponent $\frac13$.

\subsection{Where the hypothesis is used, and weaker hypotheses}\label{sec:whereused}

Hypothesis \ref{hyp:R*} is used only in Lemma \ref{lem:smoothsalie}, through Lemma \ref{lem:shortsalie}, at the moduli $m$ and $8m$ for divisors $m$ of the denominators $r\le\tau$. Without it, the method does not seem to give more than the elementary count. For instance, with only a P\'olya--Vinogradov type input $S(u)\ll m^{1/2+\eps}(A,m)^{1/2}$ for the partial sums (which we have not verified at even moduli), one gets $|\Omega_1^{h\ne0}|\ll W\,r^\eps$ in the proof of Proposition \ref{prop:F}, and hence $P_{Q_1}(\alpha)\ll(1+Q_1^3\Delta+Q_1^2r\Delta)r^\eps$, the order of the elementary count. So the whole gain comes from the hypothesis.

On the other hand, the full strength of square-root cancellation is not needed for a power saving. Suppose that, for some $0\le\lambda<\frac14$, Hypothesis \ref{hyp:R*} holds with $(x_2-x_1+1)^{1/2+\lambda}$ in place of $(x_2-x_1+1)^{1/2}$. Then the proof of Lemma \ref{lem:smoothsalie} gives the bound $T^{1/2+\lambda}m^\eps(A,m)^{1/2}$: in its two integrals one now has $|S(u)|\ll u^{1/2+\lambda}m^\eps(A,m)^{1/2}$ for $u\le m$ and $|S(u)|\ll u\,m^{-1/2+\lambda+\eps}(A,m)^{1/2}$ for $u>m$, which give $T^{1/2+\lambda}$ and $T^2m^{-3/2+\lambda}\le T^{1/2+\lambda}$. In Proposition \ref{prop:hsum} the class $d$ then contributes $\ll d\sqrt m\,(m/W)^{1/2+\lambda}m^\eps(l/d,m)\le r^{1+\lambda}W^{-1/2-\lambda}m^\eps(l/d,m)$. In Proposition \ref{prop:F} this gives $|\Omega_1^{h\ne0}|\ll r^\lambda W^{1/2-\lambda}r^{2\eps}\le\tau^\lambda W^{1/2-\lambda}r^{2\eps}$. At $N=Q^3$, Case 2 of the proof of Theorem \ref{thm:P} then gives $P(\alpha)\ll\tau^{1/2}Q^{\lambda-1/2+\eps}$, and balancing with Case 1, $Q^2/\tau$, gives
\[
\tau=Q^{(5-2\lambda)/3},\qquad P(\alpha)\ll Q^{(1+2\lambda)/3+\eps},\qquad S(Q,M,Q^3,(a_n))\ll Q^{1/2-(1-4\lambda)/6+\eps}NZ .
\]
This is a power saving over the exponent $\tfrac12$ for every $\lambda<\frac14$.

\subsection{Function fields}\label{sec:ff}

A function-field version of the large sieve with square moduli was established by Baier and Singh \cite{BaierSingh2022}, and a function-field analogue of Hooley's Hypothesis $R^*$ has been proved, for squarefree moduli, by Sawin and Shusterman \cite{SawinShusterman2025}. The proof above uses only finite Fourier analysis, the evaluation of quadratic Gauss sums, Dirichlet approximation and Hypothesis \ref{hyp:R*}, all of which have function-field analogues. This raises the prospect of an unconditional function-field analogue of Theorem \ref{thm:main}. The moduli that occur here are arbitrary, not only squarefree, and we have not checked whether the results of \cite{SawinShusterman2025} apply in the form needed.

\section{Numerical checks}\label{sec:numerics}

The following checks were carried out with short scripts, available from the author. They are evidence, not proof; every claim of the paper is proved in full or explicitly flagged.

\begin{itemize}
\item \emph{Gauss sums.} Proposition \ref{prop:classeval} was checked in the explicit form (3.8) over $32{,}167$ configurations (Section \ref{sec:gaussdegen}), and Lemmas \ref{lem:twoadic} and \ref{lem:degauss} exhaustively in the ranges stated there.
\item \emph{The double Poisson identity.} The identity $\Omega=\Omega_0+\Omega_1$ in the proof of Proposition \ref{prop:F} was checked with Gaussian weights in place of $w$ and $V$ (for which the identity holds with the same proof). This covered $10$ random configurations with odd and even moduli $r$ between $9$ and $36$; the largest relative error was $2.1\cdot10^{-11}$.
\item \emph{The elementary count.} The first inequality of Lemma \ref{lem:E} was checked in exact rational arithmetic in $400$ random configurations, with $Q\le30$, $N$ between $Q^2$ and $Q^4$, $r\le Q^2+5$, and $|z|$ up to $\Delta$, up to $20\Delta$, or up to $1/(rQ)$. There was no violation, and the bound is attained in some cases. In $300$ further configurations with $N=Q^3$ and $Q\le40$, the slightly weaker inequality with $Q^3\Delta/|m|+1$ replaced by $2$ was checked, again with no violation.
\item \emph{The size of $\Omega_1$.} For $N=Q^3$, $Q\in\{200,400\}$, $Q_1=Q/2$ and $r$ near $Q^{1.6},Q^{1.75},Q^{1.9},Q^2$, the quantity $\Omega_1=\Omega-\Omega_0$ was computed directly, for $z=0$ and about $400$ random $b$ per modulus. The maximum of $|\Omega_1|$ over $b$ was always below $W^{1/2}$, which ranged from $9.8$ to $40$; the largest ratio was $0.91$. This is the square-root cancellation for which Hypothesis \ref{hyp:R*} is used.
\item \emph{The counting function.} For $N=Q^3$, $Q\in\{60,100,150\}$, and $r$ near $Q^{1.5},\dots,Q^2$, the maximum over $b$ of $P(b/r)$ was between $3$ and $6$. The bound of Proposition \ref{prop:F} at $z=0$ is $(r/Q)^{1/2}$, which ranged from $2.9$ to $12.7$, and the elementary bound $r/Q$ from $8.5$ to $162$. Computations of this size cannot test a bound that holds up to $Q^\eps$, and are recorded only as a consistency check.
\end{itemize}

\subsection*{Acknowledgement} The author thanks the Ramakrishna Mission Vivekananda Educational and Research Institute for excellent working conditions.

\end{document}